\documentclass[12pt, reqno]{amsart}

\usepackage{amsmath,amsfonts,amsthm,amssymb,amscd}
\usepackage{amssymb}
\usepackage[utf8]{inputenc}
\usepackage[T2A]{fontenc}
\usepackage[all]{xy}
\usepackage{enumitem}
\usepackage[english]{babel}
\usepackage{graphicx}
\usepackage[svgnames,dvipsnames]{xcolor}
\usepackage[margin=2.5cm]{geometry}

\usepackage{hyperref}
\hypersetup{
    hypertexnames=false,
    colorlinks,
    linkcolor={red!50!black},
    citecolor={blue!50!black},
    urlcolor={blue!80!black}
}

\newtheorem{theorem}[subsection]{Theorem}
\newtheorem{lemma}[subsection]{Lemma}
\newtheorem{sublemma}[subsubsection]{Lemma}

\newtheorem{corollary}[subsection]{Corollary}

\newtheorem{definition}[subsection]{Definition}

\makeatletter
\@addtoreset{subsection}{section}
\@addtoreset{equation}{section}
\@addtoreset{figure}{section}
\@addtoreset{table}{section}
\makeatother

\makeatletter
\newcommand\testshape{family=\f@family; series=\f@series; shape=\f@shape.}
\def\myemphInternal#1{\if n\f@shape%
\begingroup\itshape #1\endgroup\/%
\else\begingroup\sf\itshape #1\endgroup%
\fi}
\def\myemph{\futurelet\testchar\MaybeOptArgmyemph}
\def\MaybeOptArgmyemph{\ifx[\testchar \let\next\OptArgmyemph
                 \else \let\next\NoOptArgmyemph \fi \next}
\def\OptArgmyemph[#1]#2{\index{#1}\myemphInternal{#2}}
\def\NoOptArgmyemph#1{\myemphInternal{#1}}
\makeatother

\newcommand\id{\mathrm{id}}          
\newcommand\Int{\mathrm{Int}}        
\newcommand\Fix[1]{\mathrm{Fix}(#1)} 

\newcommand\wrm[1]{\mathop{\wr}\limits_{#1}} 

\newcommand\bR{\mathbb{R}}
\newcommand\bZ{\mathbb{Z}}
\newcommand\bN{\mathbb{N}}

\newcommand\Diff{\mathcal{D}}       
\newcommand\Orb{\mathcal{O}}        
\newcommand\Stab{\mathcal{S}}       

\newcommand\DiffId{\Diff_{\id}}     
\newcommand\StabId{\Stab_{\id}}     

\newcommand\DiffPl{\Diff^{+}}      

\newcommand\Cinfty{\mathcal{C}^{\infty}}
\newcommand\Ci[2]{\mathcal{C}^{\infty}(#1,#2)}               
\newcommand\Morse[2]{\mathcal{M}(#1,#2)}                     

\newcommand\Torus{T^2}             

\newcommand\PrjPlane{\bR{P}^2}    

\newcommand\Stabilizer[1]{\Stab(#1)}
\newcommand\StabilizerId[1]{\StabId(#1)}
\newcommand\StabilizerIsotId[1]{\Stab'(#1)}
\newcommand\Orbit[1]{\Orb(#1)}

\newcommand\FolStab{\Delta} 
\newcommand\FolStabilizer[1]{\FolStab(#1)}
\newcommand\FolStabilizerNbh[1]{\FolStab_{{\nb}}(#1)}
\newcommand\FolStabilizerIsotId[1]{\FolStab'(#1)}
\newcommand\FolStabilizerNbhIsotId[1]{\FolStab'_{{\nb}}(#1)}

\newcommand\GKR{\mathbf{G}}
\newcommand\GrpKR[1]{\GKR(#1)}
\newcommand\GrpKRIsotId[1]{\GKR'(#1)}
\newcommand\GrpKRNbh[1]{\GKR_{{\nb}}(#1)}
\newcommand\GrpKRNbhIsotId[1]{\GKR'_{{\nb}}(#1)}

\newcommand\tDiff{\cov{\Diff}}            
\newcommand\tDiffPl{\cov{\Diff}^{+}}      
\newcommand\tDiffId{\cov{\Diff}_{\id}}    

\newcommand\tStab{\cov{\Stab}}       
\newcommand\tStabilizer[1]{\tStab(#1)}

\newcommand\tStabilizerNbh[1]{\tStab_{\nb}(#1)}

\newcommand\SO{\mathrm{SO}}

\newcommand\Kman{K}

\newcommand\Mman{M}
\newcommand\Nman{N}

\newcommand\Pman{P}
\newcommand\Qman{Q}

\newcommand\Uman{U}
\newcommand\Vman{V}
\newcommand\Wman{W}
\newcommand\Xman{X}
\newcommand\Yman{Y}

\newcommand\cov[1]{\widetilde{#1}} %
\newcommand\scov[1]{\tilde{#1}} %

\newcommand\DiffIdM{\DiffId(\Mman)}

\newcommand\func{f}

\newcommand\dif{h}

\newcommand\tdif{\scov{\dif}}

\newcommand\fSing{\Sigma_{\func}}

\graphicspath{{pics/}}

\newcommand\ocolor[1]{#1}
\newcommand\ncolor[1]{#1}

\newcommand   \Surf{\ocolor{\Mman}}

\newcommand   \OSurf{\ocolor{\Mman}}
\newcommand\OSubSurf{\ocolor{\Vman}}
\newcommand\OSubMan {\ocolor{\Xman}}
\newcommand      \ox{\ocolor{x}}
\newcommand    \odif{\ocolor{q}}
\newcommand   \ofunc{\ocolor{g}}

\newcommand   \NSurf{\ncolor{\Nman}}
\newcommand\NSubSurf{\ncolor{\Wman}}
\newcommand \NSubMan{\ncolor{\Yman}}
\newcommand      \nx{\ncolor{y}}
\newcommand    \ndif{\ncolor{h}}
\newcommand   \nfunc{\ncolor{\func}}

\newcommand\Cyl{\ocolor{A}}
\newcommand\MBand{\ncolor{B}}

\newcommand\XXi[1]{\ocolor{\Xman}_{#1}}
\newcommand\YYi[1]{\ncolor{\Yman}_{#1}}

\newcommand\dimM{{\color{WildStrawberry}m}}

\newcommand\KerSAct{\Qman_{\func}}

\newcommand\Gf{\Gamma(\func)}

\newcommand{\nb}{\mathrm{nb}}
\newcommand\DiffNbh{\Diff_{{\nb}}}     

\newcommand\StabilizerNbh[1]{\Stab_{{\nb}}(#1)}  
\newcommand\StabilizerNbhIsotId[1]{\Stab'_{{\nb}}(#1)}  

\newcommand\OrbitPathComp[2]{\Orb_{#2}(#1)}      

\newcommand\Stabf{\Stabilizer{\func}}

\newcommand\FSp[1]{\mathcal{F}(#1,\Pman)}

\newcommand\CrComp{K}

\newcommand\CompSet{\mathbf{Y}}
\newcommand\PlCompSet{\hat{\CompSet}}

\newcommand\DAFunc[1]{\Theta(#1)}
\newcommand\DFunc{\DAFunc{\fld}}

\newcommand\Sh[1]{\varphi_{#1}}

\newcommand\crLev{K}

\newcommand\NX{\Uman_{\NSubMan}}
\newcommand\NtX{\Uman_{\OSubMan}}

\newcommand\fld{F}

\newcommand\flow{\mathbf{F}}

\newcommand\orb{\omega}

\newcommand\ST[1]{\mathcal{P}_{\func}(#1)}

\newcommand\bdEdge{$(\mathcal{B})$}
\newcommand\nullEdge{$(\mathcal{N})$}

\title{Homotopy properties of smooth functions on the M{\"o}bius band}
\author{Iryna Kuznietsova, Sergiy Maksymenko}
\address{Department of Algebra and Topology, Institute of Mathematics of NAS of Ukraine, Tereshchenkivska str. 3, Kyiv, 01601, Ukraine}
\curraddr{}
\email{kuznietsova@imath.kiev.ua, maks@imath.kiev.ua}


\begin{document}

\begin{abstract}
Let $B$ be a M\"obius band and $f:B \to \mathbb{R}$ be a Morse map taking a constant value on $\partial B$, and $\mathcal{S}(f,\partial B)$ be the group of diffeomorphisms $h$ of $B$ fixed on $\partial B$ and preserving $f$ in the sense that $f\circ h = f$.
Under certain assumptions on $f$ we compute the group $\pi_0\mathcal{S}(f,\partial B)$ of isotopy classes of such diffeomorphisms.
In fact, those computations hold for functions $f:B\to\mathbb{R}$ whose germs at critical points are smoothly equivalent to homogeneous polynomials $\mathbb{R}^2\to\mathbb{R}$ without multiple factors.

Together with previous results of the second author this allows to compute similar groups for certain classes of smooth functions $f:N\to\mathbb{R}$ on non-orientable surfaces $N$.
\end{abstract}

\maketitle

\section{Main result}
Let $\Surf$ be a smooth compact surface, i.e.\! a 2-dimensional manifold, which can be disconnected, non-orientable, and have a non-empty boundary, and $\Pman$ be either a real line $\bR$ or a circle $S^1$.
Then the group $\Diff(\Surf)$ of $\Cinfty$-diffeomorphisms of $\Surf$ acts from the right on the space of smooth maps $\Cinfty(\Surf,\Pman)$ defined by the following rule: the result of the action of a diffeomorphism $\dif\in \Diff(\Surf)$ on $\func\in\Ci{\Surf}{\Pman}$ is the composition $f\circ h$.
Then for each $\func\in\Ci{\Surf}{\Pman}$ one can define the \textit{stabilizer} of $f$
\[ \Stabilizer{\func}=\{ \dif\in \Diff(\Surf)\ |\  \func\circ\dif=\func \} \]
and its \textit{orbit}
\[ \Orbit{\func}=\{\func\circ\dif \mid \dif\in \Diff(\Surf)\}\]
with respect to the above action.

More generally, denote by $\Diff(\Surf,\Xman)$ the group of diffeomorphisms of $\Surf$ fixed on a closed subset $\Xman\subset\Surf$.
Let also
\begin{align*}
 \Stabilizer{\func,\Xman} &= \Stabf\cap\Diff(\Surf,\Xman) &
& \text{and}&
 \Orbit{\func,\Xman} &= \{\func\circ\dif \mid \dif\in \Diff(\Surf,\Xman)\}.
\end{align*}
We will endow $\Diff(\Surf,\Xman)$ and $\Ci{\Surf}{\Pman}$ with Whitney $C^\infty$-topologies and their subspaces $\Stabilizer{\func,\Xman}$ and $\Orbit{\func,\Xman}$ with induced ones.
Then they yield certain topologies on the stabilizers and orbits of maps $\func\in \Cinfty(\Surf,\Pman)$.
Let also $\DiffId(\Surf,\Xman)$ and $\StabilizerId{\func,\Xman}$ be the identity path components of $\Diff(\Surf,\Xman)$ and $\Stabilizer{\func,\Xman}$, and $\OrbitPathComp{\func,\Xman}{\func}$ be the path component of $\Orbit{\func,\Xman}$ containing $\func$.
If $\Xman=\varnothing$, then we will omit $\Xman$ from notation.

In the present paper we continue study of the homotopy types of $\Stabilizer{\func,\Xman}$ and $\Orbit{\func,\Xman}$, see below for references and the history of the problem.
Our main results, Theorems~\ref{th:unique_cr_level} and~\ref{th:pi0S_struct}, concern with the group $\pi_0\Stabilizer{\func,\Xman}$ for the case when $\Surf$ is a M\"obius band, $\Xman=\partial\Surf$, and $\func:\Surf\to\Pman$ belongs to the following space of maps~$\FSp{\Surf}$.

\begin{definition}\label{def:classF}
Let $\FSp{\Surf}$ be the subset of $\Ci{\Surf}{\Pman}$ consisting of maps $\func:\Surf\to\Pman$ having the following properties:
\begin{enumerate}[label={\rm(\arabic*)}]
\item\label{enum:classF:const_by_bd}
the map $\func$ takes constant values at each connected component of $\partial\Surf$ and has no critical points on it;
\item\label{enum:classF:loc_equ_poly}
for every critical point $z$ of $\func$ the germ of $\func$ at $z$ is $\Cinfty$ equivalent to some homogeneous polynomial $v\colon \bR^2 \to \bR$ without multiple factors.
\end{enumerate}

A map $\func\in\Ci{\Surf}{\Pman}$ will be called \myemph{Morse}, if it satisfies condition~\ref{enum:classF:const_by_bd} and all its critical points are non-degenerate.
Denote by $\Morse{\Surf}{\Pman}$ the space of all Morse maps.
A Morse map $\func$ is \myemph{generic}, if it takes distinct values at distinct critical points.
\end{definition}
Since the polynomial $\pm x^2 \pm y^2$ is homogeneous and has no multiple factors, it follows from Morse lemma that
\[
   \Morse{\Surf}{\Pman} \subset \FSp{\Surf}.
\]
Also notice that every $\func\in\FSp{\Surf}$ has only isolated critical points.
A structure of level set foliations near critical points of $\func\in\FSp{\Surf}$ is illustrated in Figure~\ref{fig:isol_crit_pt}.
A critical point of $\func\in\FSp{\Mman}$ which is not a local extreme will be called a \myemph{saddle}.
\begin{figure}[ht]
\begin{center}
\footnotesize
\begin{tabular}{ccccccc}
\includegraphics[height=1.5cm]{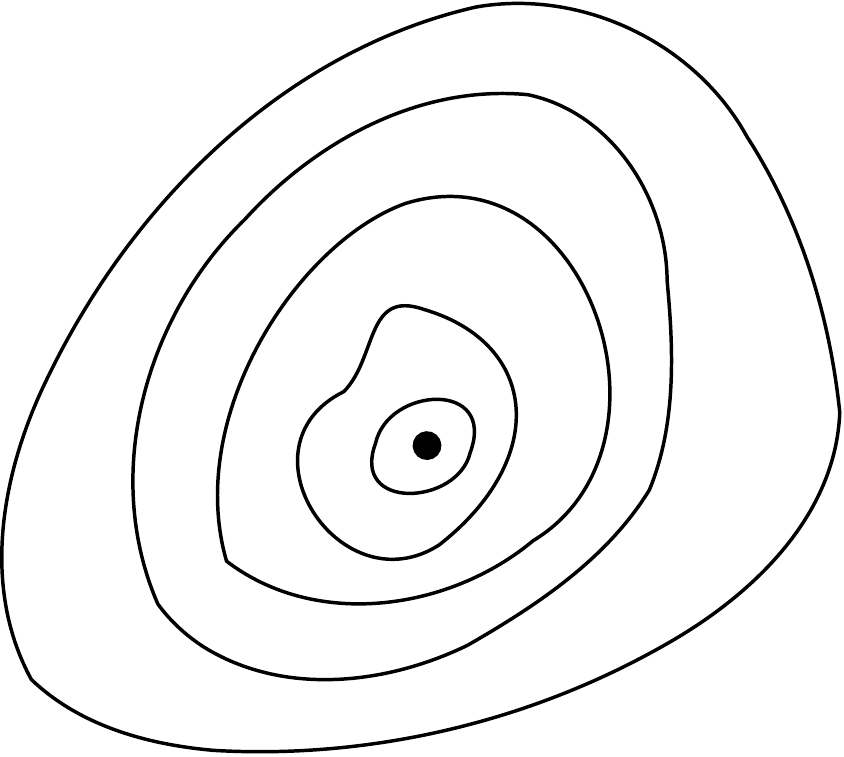} & \qquad\qquad &
\includegraphics[height=1.5cm]{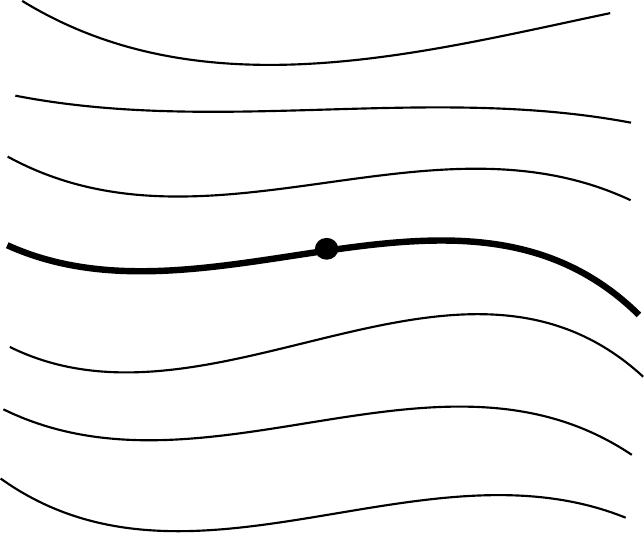}   & \qquad\qquad &
\includegraphics[height=1.5cm]{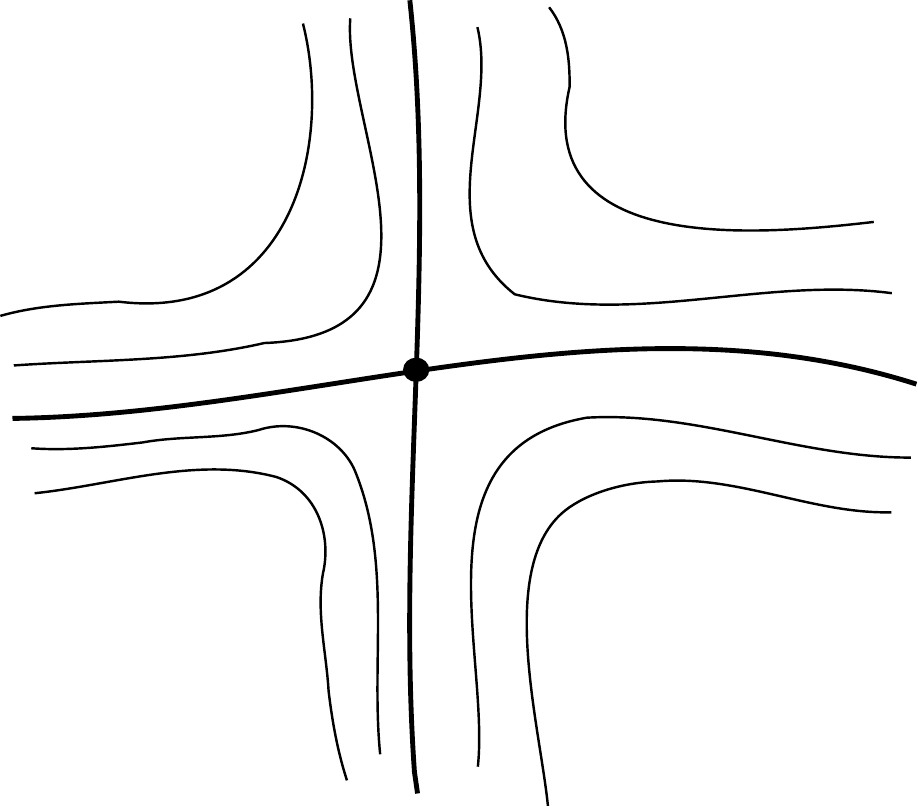}   & \qquad\qquad &
\includegraphics[height=1.5cm]{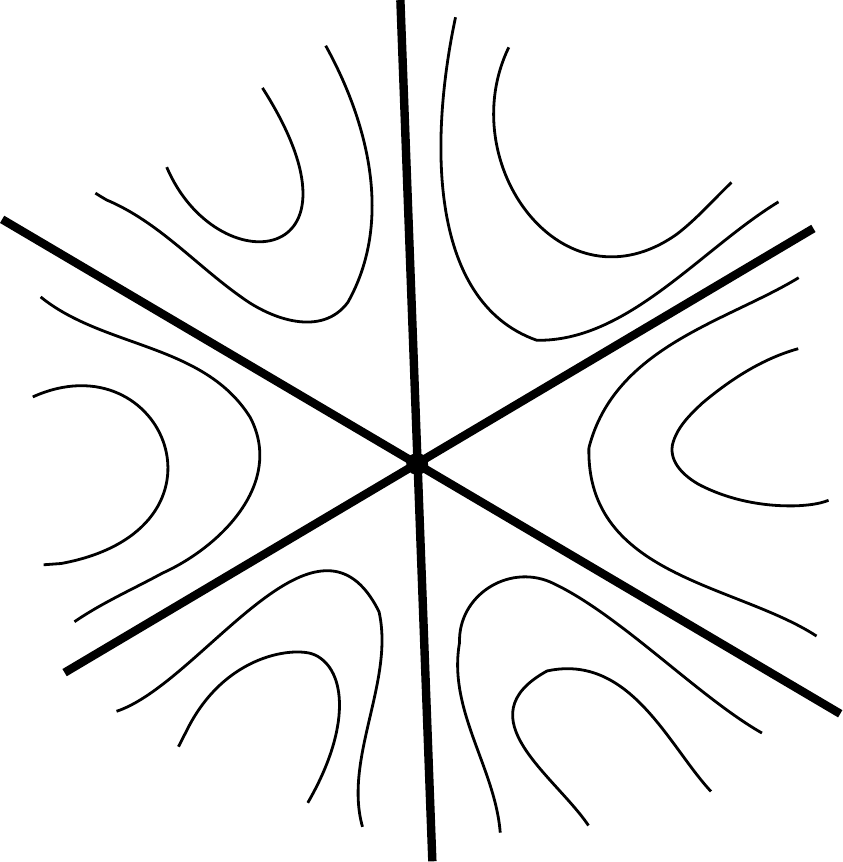}   \\
local extreme & &
              & &
saddles       & &
\end{tabular}
\caption{Topological structure of level-sets of maps from $\FSp{\Surf}$ near critical points}
\label{fig:isol_crit_pt}
\end{center}
\end{figure}

Let $\func\in \FSp{\Surf}$, $c\in\bR$, and $\Kman$ be a connected component of the level-set $\func^{-1}(c)$.
Then $\Kman$ will be called \myemph{regular} whenever it contains no critical points, and \myemph{critical} otherwise.
A connected component $\Uman$ of $\func^{-1}[c-\varepsilon, c+\varepsilon]$ containing $\Kman$ will be called an \myemph{$\func$-regular neighborhood} of $\Kman$ if $\Uman\setminus\Kman$ contains no critical points and does not intersect $\partial\Surf$.

Let $\Uman=\Uman_1 \sqcup \Uman_2 \sqcup \dots \sqcup \Uman_k$ be a disjoint union of connected one- and two-dimensional submanifolds of $\Surf$.
We will say that $\Uman$ is an \myemph{$\func$-adopted} submanifold if for each $i=1,\ldots,k$ the following conditions hold:
\begin{enumerate}
\item if $\dim\Uman_i = 1$, then $\Uman_i$ is a regular connected component of some level-set $\func^{-1}(c)$, $c\in\bR$;
\item if $\dim\Uman_i = 2$, then the connected components of the boundary $\partial\Uman_i$ are regular connected components of some level-sets of $\func$.
\end{enumerate}
In particular, an $\func$-regular neighborhood is an $\func$-adopted \myemph{subsurface}.
Evidently, if $\Uman$ is adopted subsurface, then the restriction $\func|_{\Uman}$ belongs to the space $\FSp{\Uman}$.

Denote
\[
\StabilizerIsotId{\func,\Xman} = \Stabilizer{\func} \cap\DiffId(\Surf,\Xman).
\]
The following statement collects known information about the homotopy types of stabilizers and orbits of $\func\in\FSp{\Surf}$.

\begin{theorem}\label{th:stab_orb_full_info}
Let $\Surf$ be a connected compact surface, $\func\in\FSp{\Surf}$, and $\Xman$ be a union of finitely many connected components of some level-sets of $\func$ and some critical points of $\func$.
Then the following statements hold.

\begin{enumerate}[leftmargin=*, label={\rm(\arabic*)}, wide, itemsep=0.6ex]
\item\label{enum:th:stab_orb_full_info:Serre}
The map $p:\Diff(\Surf,\Xman)\to\Orbit{\func,\Xman}$ defined by $p(\dif) = \func\circ\dif$ is a locally trivial principal $\Stabilizer{\func,\Xman}$-fibration.
In particular, the restriction $p:\DiffId(\Surf,\Xman)\to\OrbitPathComp{\func,\Xman}{\func}$ is a locally trivial principal $\StabilizerIsotId{\func,\Xman}$-fibration,
\cite{Sergeraert:ASENS:1972}, \cite{Maksymenko:AGAG:2006}.

\item\label{enum:th:stab_orb_full_info:Sidf}
The group $\StabilizerId{\func,\Xman}$ is homotopy equivalent to the circle if and only if the following condition holds:
\begin{itemize}[label=$\bullet$]
\item $\Surf$ is orientable, $\chi(\Surf)\geq0$, $\Xman$ is a collection of at most $\chi(\Surf)$ critical points of $\func$, and each critical point of $\func$ is a \myemph{non-degenerate local extreme}.
\end{itemize}
Otherwise, $\StabilizerId{\func,\Xman}$ is contractible, and in this case
\begin{enumerate}[leftmargin=9ex, label={\rm(\alph*)}, itemsep=0.6ex]
\item\label{enum:Of:S2_min_max}
if $\Mman=S^2$, $\Xman=\varnothing$, and $\func$ is Morse having exactly two critical points (minimum and maximum), then $\OrbitPathComp{\func}{\func}$ is homotopy equivalent to $S^2$;
\item\label{enum:Of:S2_RP2}
otherwise, if $\Mman=S^2$ or $\bR{P}^2$, and $\Xman=\varnothing$, then $\pi_k\OrbitPathComp{\func}{\func} \cong \pi_k S^3$ for $k\geq2$;

\item\label{enum:Of:aspherical}
otherwise, $\pi_k\OrbitPathComp{\func,\Xman}{\func} = 0$ for $k\geq2$, \cite{Maksymenko:AGAG:2006}, \cite{Maksymenko:ProcIM:ENG:2010}.
\end{enumerate}

\item\label{enum:th:stab_orb_full_info:Sid_contr}
Suppose $\StabilizerId{\func,\Xman}$ is contractible.
\begin{equation}\label{equ:exact_seq_for_pi1OfX}
1 \to \pi_1\DiffId(\Surf,\Xman) \xrightarrow{~p~} \pi_1\Orbit{\func,\Xman} \to \pi_0\StabilizerIsotId{\func,\Xman} \to 1.
\end{equation}
If $\chi(\Surf) < |\Xman|$, the group $\DiffId(\Surf,\Xman)$ is contractible as well, and~\eqref{equ:exact_seq_for_pi1OfX} yields an isomorphism
\[\pi_1\OrbitPathComp{\func,\Xman}{\func} \cong \pi_0\StabilizerIsotId{\func,\Xman},\]
see \cite{Maksymenko:AGAG:2006}, \cite{Maksymenko:ProcIM:ENG:2010}.

\item\label{enum:th:stab_orb_full_info:Off_V}
$\OrbitPathComp{\func,\Xman}{\func} = \OrbitPathComp{\func,\Xman \cup \Vman}{\func}$ for any union of boundary components $\Vman$ of $\Surf$, \cite{Maksymenko:UMZ:ENG:2012}.

\item\label{enum:th:stab_orb_full_info:Of_braid}
If $\func$ is Morse and has exactly $n$ critical points, then $\OrbitPathComp{\func}{\func}$ is homotopy equivalent to a certain covering space of the $n$-th configuration space of $\Surf$, which in turn is homotopy equivalent to some (possibly non-compact) $(2n-1)$-dimensional CW-complex.
In particular, $\pi_1\OrbitPathComp{\func}{\func}$ is a subgroup of the $n$-th braid group $\mathcal{B}_{n}(\Surf)$ of $\Surf$, \cite{Maksymenko:TrMath:2008}.

\item\label{enum:th:stab_orb_full_info:Of_generic}
Suppose $\func$ is \myemph{generic}.
If $\Surf=S^2$ and $\func$ has exactly two critical points being local extremes, then $\OrbitPathComp{\func}{\func}$ is homotopy equivalent to $S^2$.
Otherwise, if $\Surf=S^2$ or $\PrjPlane$, then $\OrbitPathComp{\func}{\func}$ is homotopy equivalent to $\SO(3) \times (S^1)^k$ for some $k\geq0$.
In all other cases $\OrbitPathComp{\func}{\func}$ is homotopy equivalent to $(S^1)^k$ for some $k\geq0$, \cite{Maksymenko:AGAG:2006}.

\item\label{enum:th:orb_full_info}
Suppose $\Surf$ is orientable, $\func\in\Morse{\Surf}{\bR}$, and $\chi(\Surf) < |\Fix{\StabilizerIsotId{\func}}|$ (which holds e.g.\! if $\chi(\Surf)<0$ or if $\func$ is generic and has at least one saddle critical point).
Then $\OrbitPathComp{\func}{\func}$ has the homotopy type of the quotient $(S^1)^k/ {G}$ of $(S^1)^k$ by a free action of some finite group ${G}$
if $\Surf\not=S^2$, and the homotopy type of $((S^1)^k/{G}) \times \SO(3)$ if $\Surf=S^2$, \cite{Kudryavtseva:SpecMF:VMU:2012}, \cite{Kudryavtseva:MathNotes:2012}, \cite{Kudryavtseva:MatSb:ENG:2013}, and also~\cite{Kudryavtseva:ENG:DAN2016} for extensions to functions with prescribed local singularities of $A_{\mu}$-types, $\mu\in\bN$.
\end{enumerate}
\end{theorem}

Results in~\ref{enum:th:orb_full_info} are obtained by E.~Kudryavtseva.

Notice that in the case~\ref{enum:Of:aspherical}, e.g. when if $\Surf$ is distinct from $2$-sphere and projective plane, then $\OrbitPathComp{\func}{\func}$ is \myemph{aspherical}, and so its homotopy type is completely determined by the fundamental group $\pi_1\OrbitPathComp{\func}{\func}$.

If $\func$ is generic, then by~\ref{enum:th:stab_orb_full_info:Of_generic} $\OrbitPathComp{\func}{\func}$ is homotopy equivalent to some torus $(S^1)^k$, whence
$\pi_1\OrbitPathComp{\func}{\func} = \bZ^k$ is free abelian.

Suppose $\Surf$ is orientable and differs from $S^2$.
Then by~\ref{enum:th:orb_full_info} we have a certain free action of a finite group ${G}$ on the torus $(S^1)^k$.
Hence the quotient map $q:(S^1)^k \to (S^1)^k/{G}$ is a locally covering map, whence we have the following short exact sequence:
\[
1 \to \pi_1 (S^1)^k \to \pi_1 (S^1)^k /{G} \to {G} \to 1,
\]
which due to~\ref{enum:th:orb_full_info} can be rewritten as follows:
\[
1 \to \bZ^k \to \pi_1\OrbitPathComp{\func}{\func} \to {G} \to 1.
\]
This sequence was first discovered in~\cite[Eq.~(1.6)]{Maksymenko:AGAG:2006}.
In particular, it implies that $\pi_1\OrbitPathComp{\func}{\func}$ is a \myemph{crystallographic} group, i.e.\! contains a free abelian normal subgroup of finite index.
Moreover, due to~\ref{enum:th:stab_orb_full_info:Of_braid} $\pi_1\OrbitPathComp{\func}{\func}$ is also a subgroup of a certain braid group $\mathcal{B}_{n}(\Surf)$ of $\Surf$.
Since $\mathcal{B}_{n}(\Surf)$ has no elements of finite order, so does $\pi_1\OrbitPathComp{\func}{\func}$, and therefore it is a \myemph{Bieberbach} group.

To describe further known results, for every $\func$-adopted connected subsurface $\Xman \subset \Surf$ let
\[
\ST{\Xman} := \pi_1\OrbitPathComp{\func|_{\Xman}}{\func|_{\Xman}}.
\]
be the fundamental group of the orbit of the restriction of $\func$ to $\Xman$.
In particular, if either $\partial\Surf$ is non-empty or $\chi(\Surf)<0$, then we get from Theorem~\ref{th:stab_orb_full_info} the following isomorphisms:
\begin{equation}\label{equ:pi0stab_pi1orb}
    \ST{\Surf} := \pi_1\OrbitPathComp{\func}{\func}
      \ \stackrel{\ref{enum:th:stab_orb_full_info:Off_V}}{\cong} \
    \pi_1\OrbitPathComp{\func,\partial\Surf}{\func}
    \ \stackrel{\ref{enum:th:stab_orb_full_info:Sid_contr}}{\cong} \
    \pi_0\StabilizerIsotId{\func, \partial\Surf}.
\end{equation}

The following statement summarizes several results about $\ST{\Surf}$.

\begin{theorem}\label{th:Stabf_reduction}
{\rm
\cite{Maksymenko:MFAT:2010},
\cite{MaksymenkoFeshchenko:UMZ:ENG:2014},
\cite{MaksymenkoFeshchenko:MS:2015},
\cite{MaksymenkoFeshchenko:MFAT:2015},
\cite{Feshchenko:Zb:2015},
\cite{Feshchenko:MFAT:2016}.}
Let $\Surf$ be a compact surface.
Then for each $\func\in\FSp{\Surf}$ there exist mutually disjoint $\func$-adopted subsuraces
$\YYi{1},\ldots,\YYi{n} \subset\Surf$ having the following properties.
\begin{enumerate}[leftmargin=*, label={\rm(\arabic*)}]
\item\label{enum:PM:chi_neg}
If $\chi(\Surf)<0$, then each $\YYi{i}$ is either a $2$-disk or an annulus or a M\"obius band, and
\[ \ST{\Surf} \cong \prod_{i=1}^{n} \ST{\YYi{i}}. \]
\item\label{enum:PM:torus}
Suppose $\Surf = \Torus$ is a $2$-torus.
\begin{enumerate}[leftmargin=*, label={\rm(\alph*)}]
\item
If the Kronrod-Reeb graph $\Gf$ of $\func$ (see~\S\ref{sect:KRGraph} for definition) is a tree, then each $\YYi{i}$ is a $2$-disk, and
\[
\ST{\Torus} \cong \bigl( \prod_{k=1}^{n} \ST{\YYi{k}} \bigr)\wrm{a,b} \bZ^2.
\]
for some $a,b\geq1$.

\item
Otherwise, $\Gf$ contains a unique cycle, $n=1$, $\YYi{1}$ is an annulus, and
\[ \ST{\Torus} \cong \ST{\YYi{1}} \wrm{k} \bZ\]
for some $k\geq1$.
\end{enumerate}
\end{enumerate}
\end{theorem}

Here $A\wrm{a,b}\bZ^2$ and $A\wrm{k}\bZ$ are \myemph{wreath} products of certain types, which are not essential for our considerations.
Notice that in this theorem $\ST{\Surf}$ and $\ST{\YYi{i}}$ can be replaced by either of the groups of type~\eqref{equ:pi0stab_pi1orb}.
On the other hand, $\ST{\Torus} = \pi_1\OrbitPathComp{\func}{\func}$ is not the same as $\pi_0\StabilizerIsotId{\func}$ since $\pi_1\Diff(\Torus)\cong \bZ^2$ and due to~\eqref{equ:exact_seq_for_pi1OfX} we have the following short exact sequence: $1 \to \bZ^2 \to \pi_1\OrbitPathComp{\func}{\func} \to \pi_0\StabilizerIsotId{\func} \to 1.$

Theorem~\ref{th:Stabf_reduction} shows that for most surfaces computation of $\ST{\Surf}$ reduces to the cases of $2$-disk, annulus and M\"obius band.
Structure of $\ST{\Surf}$ for $\Surf$ being $2$-disk and annulus is completely described in~\cite{Maksymenko:DefFuncI:2014}.
The remained open cases are $2$-sphere and all non-orientable surfaces.

Our aim is to describe the structure of $\ST{\Surf}$ for the case when $\Surf$ is a M\"obius band and under certain restrictions on $\func\in\FSp{\Surf}$.

\begin{theorem}\label{th:unique_cr_level}
Let $\MBand$ be a M\"obius band and $\nfunc\in\FSp{\MBand}$.
There exists a unique critical component $\CrComp$ of some level-set of $\nfunc$ with the following property:
if $\NSubSurf$ is an $\nfunc$-regular neighborhood of $\CrComp$ and $\YYi{0}, \YYi{1}, \dots, \YYi{n}$ are all the connected components of $\overline{\MBand\setminus\NSubSurf}$ enumerated so that $\partial\MBand\subset\YYi{0}$, then $\YYi{0}$ is an annulus $S^1\times[0,1]$, and each $\YYi{k}$, $k=1,\ldots,n$, is a $2$-disk.
In particular,
\begin{align*}
    \ndif(\CrComp)&=\CrComp,  &
    \ndif(\YYi{0})&=\YYi{0}, &
    \ndif\bigl(\mathop{\cup}\limits_{k=1}^{n}\YYi{k}\bigr) &= \mathop{\cup}\limits_{k=1}^{n}\YYi{k},
\end{align*}
for each $\ndif\in\Stabilizer{\nfunc}$.
\end{theorem}

Let $\CompSet=\{\YYi{1}, \dots, \YYi{n}\}$ be the family of all connected components of $\overline{\MBand\setminus\NSubSurf}$ being $2$-disks as in Theorem~\ref{th:unique_cr_level}.
Since $\mathop{\cup}\limits_{k=1}^{n}\YYi{k}$ is invariant with respect to $\Stabilizer{\nfunc,\partial\MBand}$, we have a natural action of $\Stabilizer{\nfunc,\partial\MBand}$ on $\CompSet$ by permutations.

Let us fix an orientation of each $\YYi{k}$, $k=1,\ldots,n$, and put $\PlCompSet = \CompSet \times\{\pm 1\}$.
Then the action of $\Stabilizer{\nfunc,\partial\MBand}$ on $\CompSet$ extends to an action on $\PlCompSet$ defined by the following rule:
if $\dif\in\Stabilizer{\nfunc,\partial\MBand}$ and $\YYi{k}\in\CompSet$, then $\dif(\YYi{k},+1)=(\dif(\YYi{k}),\delta)$ and $\dif(\YYi{k},-1)=(\dif(\YYi{k}),-\delta)$, where
\begin{equation*}
\delta=
\begin{cases}
  +1, & \text{if the restriction $\dif|_{\YYi{k}}: \YYi{k} \to \dif(\YYi{k})$ preserves orientation}, \\
  -1, & \text{otherwise}.
\end{cases}
\end{equation*}
Let $\KerSAct$ be the normal subgroup of $\Stabilizer{\nfunc,\partial\MBand}$ consisting of diffeomorphisms preserving each $\YYi{k}$ with its orientation.
In other words, $\KerSAct$ is the kernel of non-effectiveness of the action of $\Stabilizer{\nfunc,\partial\MBand}$ on $\PlCompSet$.
Hence the action of the quotient $\Stabilizer{\nfunc,\partial\MBand}/\KerSAct$ on $\PlCompSet$ is effective.
However the induced action of $\Stabilizer{\nfunc,\partial\MBand}/\KerSAct$ on $\CompSet$ is not in general effective.

\begin{theorem}\label{th:pi0S_struct}
The quotient group $\Stabilizer{\func,\partial\MBand}/\KerSAct$ \myemph{freely acts} on $\PlCompSet$, and we have an isomorphism
\begin{equation}\label{equ:Qf_struct}
    \pi_0\KerSAct \, \cong \, \bZ \times \prod\limits_{i=0}^{n} \ST{\YYi{i}}.
\end{equation}
In particular, if $\Stabilizer{\nfunc,\partial\MBand} = \KerSAct$, then we have an isomorphism:
\begin{equation}\label{equ:pi0SfdM_triv_act}
\ST{\MBand} \, \cong \, \bZ \times \prod\limits_{i=0}^{n} \ST{\YYi{i}}.
\end{equation}
\end{theorem}
The case when the group $\Stabilizer{\func,\partial\MBand}/\KerSAct$ is non-trivial will be considered in another paper.

Due to~\ref{enum:PM:chi_neg} of Theorem~\ref{th:Stabf_reduction} a knowledge of $\ST{\MBand}$ will allow to compute $\ST{\Surf}$ for all non-orientable surfaces with $\chi(\Surf)<0$.
Together with results of~\cite{Maksymenko:DefFuncI:2014}, describing algebraic structure of $\ST{\Surf}$ for $\Surf$ being $2$-disk and annulus, this will give a complete description of the groups $\ST{\Surf}$ for all compact surfaces except for $2$-sphere, projective plane, and Klein bottle.
Also during the proof of Theorem~\ref{th:pi0S_struct} we will get more detailed information about $\ST{\MBand}$.

\begin{figure}[htbp]
\begin{tabular}{ccc}
\includegraphics[width=0.45\textwidth]{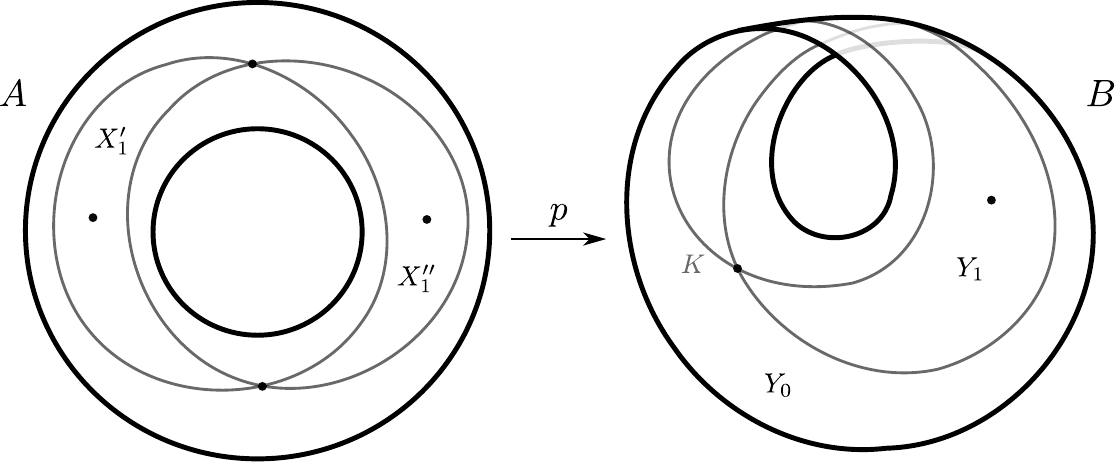} & \qquad\qquad & \includegraphics[width=0.45\textwidth]{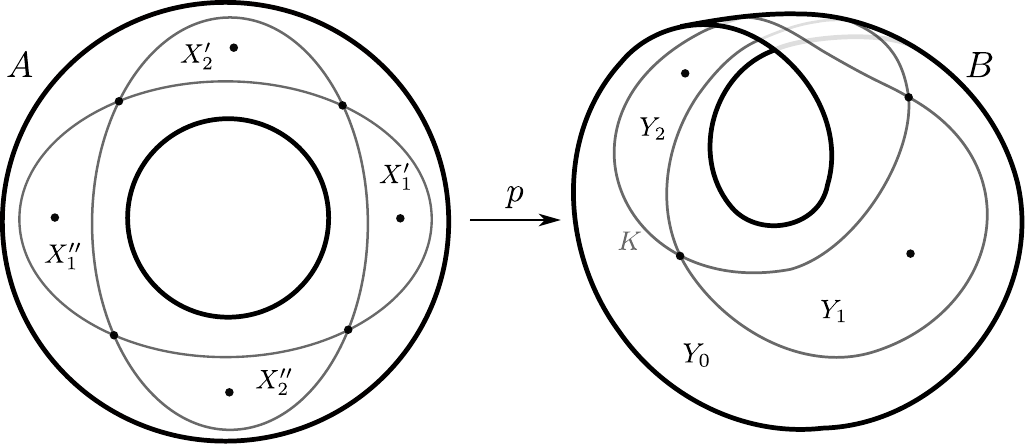} \\
(a)  && (b) \\
\includegraphics[width=0.45\textwidth]{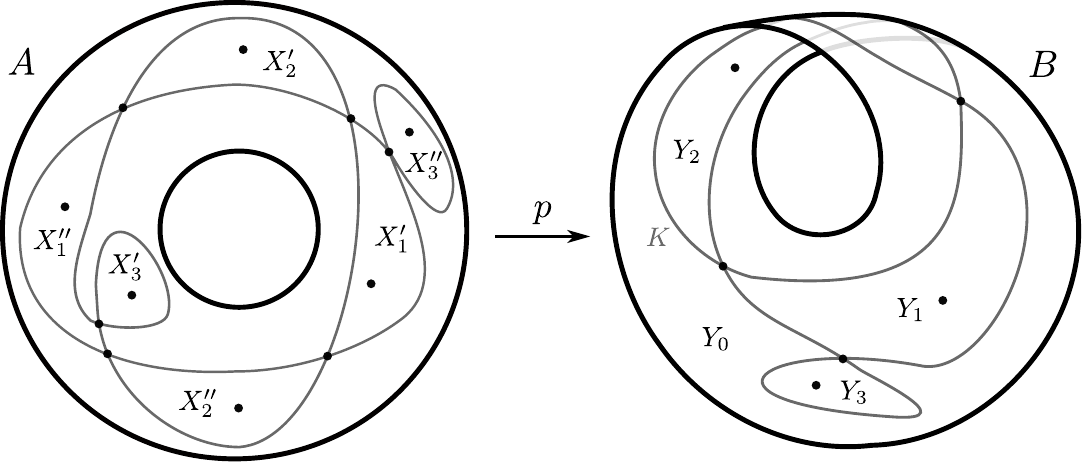} & \qquad\qquad & \includegraphics[width=0.45\textwidth]{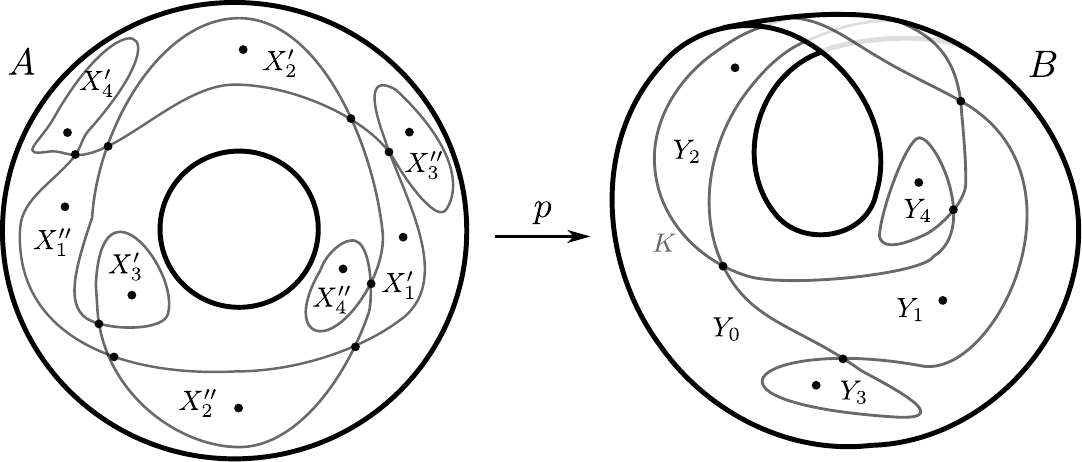} \\
(c)  && (d)
\end{tabular}
\caption{Critical level sets of certain functions on the M\"obius band}\label{fig:examples_func_mb}
\end{figure}

\subsection*{Examples}
Let $\Cyl = S^1\times[0,1]$ be an annulus, $\xi(\phi,t) = (\phi+\pi, 1-t)$ be the involution without fixed points and changing orientation of $\Cyl$, so $\MBand = \Cyl/\xi$ is a M\"obius band, and $p:\Cyl\to\MBand$ is an orientable double covering of $\MBand$.
Figure~\ref{fig:examples_func_mb} contains examples of critical components $\CrComp$ level-sets of Morse functions $\nfunc:\MBand\to\bR$ described by Theorem~\ref{th:unique_cr_level} and their preimages in $\Cyl$.
In order to simplify the illustration we denote by $\YYi{i}$ the connected components of $\MBand\setminus\CrComp$ (not of $\MBand\setminus\NSubSurf$ as in Theorem~\ref{th:unique_cr_level}), and by $\XXi{i}'$ and $\XXi{i}''$ connected components of $p^{-1}(\YYi{i})$ for $i\geq1$.

\begin{enumerate}[wide, label={Case (\alph*).}]
\item
There exists $\ndif\in\Stabilizer{\nfunc,\partial\MBand}$ such that $\ndif(\YYi{1})=\YYi{1}$ and it reverses orientation of $\YYi{i}$.
Then $\Stabilizer{\nfunc,\partial\MBand}/\KerSAct \cong \bZ_2$ and it generated by the isotopy class of $\ndif$.
Moreover, the action of $\Stabilizer{\nfunc,\partial\MBand}/\KerSAct$ on $\PlCompSet$ is transitive.

\item
In this case $\Stabilizer{\nfunc,\partial\MBand}/\KerSAct \cong \bZ_4$ is generated by the isotopy class of $\ndif\in\Stabilizer{\nfunc,\partial\MBand}$ such that $\ndif(\YYi{1})=\YYi{2}$, $\ndif(\YYi{2})=\YYi{1}$ and $\ndif^2$ reverses orientations of both $\YYi{1}$ and $\YYi{2}$.
Now the action of $\Stabilizer{\nfunc,\partial\MBand}/\KerSAct$ on $\PlCompSet$ is transitive as well.

\item
Evidently, each $\ndif\in\Stabilizer{\nfunc,\partial\MBand}$ preserves each $\YYi{i}$, $i=1,2,3$, with its orientation.
This means that $\Stabilizer{\nfunc,\partial\MBand}=\KerSAct$, so the group $\Stabilizer{\nfunc,\partial\MBand}/\KerSAct$ is trivial.

\item
Now $\Stabilizer{\nfunc,\partial\MBand}/\KerSAct \cong \bZ_2$ is generated by the isotopy class of $\ndif\in\Stabilizer{\nfunc,\partial\MBand}$ such that
\begin{align*}
\ndif(\YYi{1},+)&=(\YYi{1},-), &
\ndif(\YYi{2},+)&=(\YYi{2},-), &
\ndif(\YYi{3})&=(\YYi{4}), &
\ndif(\YYi{4})&=(\YYi{3}).
\end{align*}
\end{enumerate}

\subsection*{Structure of the paper}
In~\S\ref{sect:KRGraph} we recall the notion of the Kronrod-Reeb graph of a map $\func\in\FSp{\Surf}$, and in \S\ref{sect:th:1} prove of Theorem~\ref{th:unique_cr_level}.
\S\ref{sect:diff_non_or} contains certain results about relations of diffeomorphism groups of a non-orientable manifold and its double covering.
In \S\ref{sect:ham_vh} we recall the notion of a Hamiltonian like flow for a function on an orientable surface.
In~\S\ref{sect:FolStab_Grp} we introduce several subgroups of $\Stabilizer{\func}$ and prove Theorem~\ref{th:hom_equ} allowing to ``simplify'' diffeomorphisms from the stabilizer of $\func\in\FSp{\Surf}$.
These results extend~\cite[\S3 \& \S7]{Maksymenko:DefFuncI:2014} to non-orientable case.
\S\ref{sect:func_on_annulus} describes the relation between the groups $\Stabilizer{\func,\partial\Cyl}$ and $\StabilizerIsotId{\func,\partial\Cyl}$ for functions on the annulus $\Cyl=S^1\times[0,1]$, see Lemma~\ref{lm:cyl:rel_StStIsotId}.
In~\S\ref{sect:th:2} we prove Theorem~\ref{th:pi0S_struct}.

\section{Kronrod-Reeb graph}\label{sect:KRGraph}
Let $\Mman$ be a compact surface.
Given $\func\in\FSp{\Mman}$ consider the partition of $\Mman$ into connected components of level sets of $\func$.
Let also $\Gf$ be the set of elements of that partition and $p:\Mman\to\Gf$ be the quotient map.
Endow $\Gf$ with the corresponding quotient topology, so a subset $W\subset\Gf$ is open if and only if $p^{-1}(W)$ is open in $\Mman$.

Since $\func$ takes constant values on connected components of $\partial\Mman$ and has only finitely many critical points, it follows that $\Gf$ is a ``topological graph'', i.e. a one-dimensional CW-complex.
It is also called the \textit{Kronrod-Reeb graph} or simply the \textit{graph} of $\func$.

The following statement is well known for Morse maps, and can easily be extended to maps $\Mman\to\Pman$ with isolated critical points and taking constant values at each connected component of $\partial\Mman$.
\begin{lemma}{\rm cf.~\cite[Corollary~3.8]{Rezende_Ledesma_ManzoliNeto_Vago:TA:2018}.}
Let $\func\in\FSp{\Mman}$.
Then the quotient mapping $p:\Mman\to\Gf$ induces an epimorphism $p_{*}:H_1(\Mman,\partial\Mman,\bZ)\to H_1(\Gf,\bZ)$ between the corresponding \textit{integer} homology groups.
\end{lemma}
\begin{proof}
One easily shows that there exists a continuous map $s:\Gf\to\Mman$ such that $p \circ s$ is homotopic to $\id_{\Gf}$, so $s$ is a ``\textit{homotopical section}'' of the map $p:\Mman\to\Gf$.
Hence we get the following commutative diagram
\[
\xymatrix{
& H_1(\Mman,\bZ) \ar[rd]^{p_{*}} \\
H_1(\Gf,\bZ) \ar[ru]^{s_{*}} \ar[rr]^{\id} && H_1(\Gf,\bZ)
}
\]
implying surjectivity of $p_*$.
\end{proof}
\begin{corollary}\label{cor:KRgraph_tree}
Let $\Mman$ be either a $2$-sphere or a projective plane with $k\geq0$ holes.
Then for each $\func\in\FSp{\Mman}$ the homomorphism $p_{*}$ is zero, whence the Kronrod-Reeb graph $\Gf$ of $\func$ is a tree.
\end{corollary}
\begin{proof}
Notice that for such surfaces the homomorphism $i_{*}:H_1(\partial\Mman,\bZ) \to H_1(\Mman,\bZ)$ induced by the inclusion $i:\partial\Mman \subset \Mman$ is surjective.
Since $\func$ takes constant values at boundary components of $\Mman$, it follows that $p_{*} \circ i_{*}=0$, whence $p_{*}$ is zero epimorphism.
Therefore $H_1(\Gf,\bZ)=0$ and so $\Gf$ is a tree.
\end{proof}

\section{Proof of Theorem~\ref{th:pi0S_struct}}\label{sect:th:1}
Let $\MBand$ be a M\"obius band, $\nfunc\in\FSp{\MBand}$, and $\Gf$ be the Kronrod-Reeb graph of $\nfunc$ being due to Corollary~\ref{cor:KRgraph_tree} a tree.
We have to find a connected component $\CrComp$ of some level set of $\nfunc$ satisfying statement of Theorem~\ref{th:pi0S_struct}.

Recall that up to an isotopy and changing of orientation there are exactly two classes of two-sided simple closed curves on M\"obius strip:
\begin{enumerate}[label={\rm(\Alph*)}]
\item[\bdEdge]
a curve isotopic to $\partial\MBand$ and dividing $\MBand$ into an annulus and a M{\"o}bius strip;
\item[\nullEdge]
a null-homotopic curve dividing $\MBand$ into a $2$-disk and a M{\"o}bius strip with a hole.
\end{enumerate}

In particular, each regular component $\gamma$ of each level-set of $\nfunc$ is a two-sided simple closed curve in $\MBand$, and so it has one of the above types \bdEdge\ or \nullEdge.
Notice that $p(\gamma)$ is an internal point of some open edge $e$ of $\Gf$.
If $\gamma'$ is another regular component of some level set such that $p(\gamma')\in e$, then $\gamma'$ is isotopic to $\gamma$, and therefore it has the same type \bdEdge\ or \nullEdge\ as $\gamma$.
Hence one can associate to each edge $e$ of $\Gf$ the type \bdEdge\ or \nullEdge\ being the type of $p^{-1}(w)$, where $w$ is any point in $e$.

Therefore Theorem~\ref{th:pi0S_struct} can be reformulated as follows: \textit{there exists a unique vertex $v\in\Gf$ having exactly one incident \bdEdge-edge}.
Then $\CrComp = p^{-1}(v)$.

For the proof we need the following lemma.
Denote by $v_{0} = p(\partial\MBand)$ the vertex of $\Gf$ corresponding to the boundary of $\MBand$.
\begin{lemma}\label{lm:prof_ab_edges}
\begin{enumerate}[leftmargin=4ex, label={\rm(\roman*)}, wide]
\item\label{enum:prof_ab_edges:bd}
A vertex $v\in\Gf$ can not have more than two incident \bdEdge-edges.

\item\label{enum:prof_ab_edges:null}
Let $e$ be an open \nullEdge-edge, $w\in e$ be a point, and $T_w$ be a connected component of $\Gf\setminus w$ that does not contain $v_{0}$.
Then every edge in $T_{w}$ is of type \nullEdge as well.
\end{enumerate}
\end{lemma}
\begin{proof}
\ref{enum:prof_ab_edges:bd}
Let $p^{-1}(v)$ be the critical component of some level set of $\nfunc$ corresponding to $v$,
$e_1,\ldots,e_m$ be all the \bdEdge-edges incident to $v$, $\gamma_i$, $i=1,\ldots,m$ be a connected component of a level-set of $\nfunc$ corresponding to some point of $e_i$, $Q = \MBand\setminus \cup_{i=1}^{m} \gamma_i$, and $Q_0, Q_1, \ldots, Q_k$ be all the connected components of $Q$.
One can assume that $p^{-1}(v) \subset Q_0$, whence $\cup_{i=1}^{m}\gamma_{i} \subset  \overline{Q_0}$ as well.

Now by assumption any two curves $\gamma_i$ and $\gamma_j$ are disjoint, not null homotopic, and isotopic each other.
Hence they bound an annulus $C_{ij}$ in $\MBand$ with $\partial C_{ij} = \gamma_i \cup \gamma_j$.

Assume now that $m\geq3$, so we have at least three annulus $C_{12}$, $C_{13}$ and $C_{23}$.
Then their union $Z = C_{12} \cup C_{13} \cup C_{23}$ is connected.

If the interiors of those annulus were mutually disjoint, then
\begin{align*}
  (C_{12}\setminus\gamma_1) \cap C_{23} &= \gamma_2 \not=\varnothing, &
  (C_{13}\setminus\gamma_1) \cap C_{23} &= \gamma_3 \not=\varnothing,
\end{align*}
whence
\[
  Z \setminus \gamma_1 = (C_{12}\setminus\gamma_1) \cup C_{23} \cup (C_{13}\setminus\gamma_1)
\]
would be connected which contradicts to the property that $\MBand\setminus\gamma_1$ is disconnected.

Hence, renumbering indexes if necessary, one can assume that $C_{12} \subset C_{13}$, and so $\gamma_2\subset \Int{C_{13}}$.
But $\gamma_2 \subset \overline{Q_0}$ as well, whence
\[
Q_0 \subset C_{13} \setminus (\gamma_1\cup\gamma_2\cup\gamma_3) = \Int{C_{12}} \cup \Int{C_{23}},
\]
and therefore $Q_0$ is contained either in $\Int{C_{12}}$ or in $\Int{C_{23}}$.
Assume for definiteness that $Q_0 \subset \Int{C_{12}}$.
Then $\overline{Q_0} \subset C_{12} \subset \MBand\setminus\gamma_3$ which contradict to the assumption that $\cup_{i=1}^{m}\gamma_{i} \subset \overline{Q_0}$.
Hence $m\leq 2$.

\medskip

\ref{enum:prof_ab_edges:null}
Notice that $p^{-1}(T_{w})$ is an open disk.
Hence if $e' \subset T_w$ is an open edge and $w'\in e'$ is a point, then the curve $p^{-1}(w)$ bounds in $p^{-1}(T_{w})$ a disk, and so $e'$ is of type \nullEdge.
\end{proof}

Now we can finish Theorem~\ref{th:pi0S_struct}.

First we show that such a vertex $v$ exists.
Let $v_{0} = p(\partial\MBand)$, and  $e_0=(v_0,v_1)$ be a unique edge of $\Gf$ incident to $v_0$, where $v_1$ is another vertex of $e_0$.
Evidently, $e_0$ is of type \bdEdge.
If there is no other \bdEdge-edges incident to $v_1$ except for $e_0$, then $v=v_1$ is the required vertex.

Otherwise, due to~\ref{enum:prof_ab_edges:bd} of Lemma~\ref{lm:prof_ab_edges} exists a unique \bdEdge-edge $e_1=(v_1,v_2)$ incident to $v_1$ and distinct from $e_0$.
Applying the same arguments to $e_1$ and so on we will stop (due to the finiteness of $\Gf$) at a unique path
\[
  \pi: e_0=(v_0,v_1), e_1=(v_1,v_2),\ldots,e_m=(v_m,v)
\]
of mutually distinct \bdEdge-edges such that its end vertex $v$ has a unique \bdEdge-edge.

Let us prove a uniqueness $v$.
Let $v'$ be a vertex of $\Gf$ distinct from $v$ and $k$ be the number of \bdEdge-edges incident to $v'$.
We should prove that $k=0$ or $2$.
If $v'\in \pi$, then by the construction $k=2$.

We claim that $k=0$ for all other vertices.
Indeed, let $T$ be the connected component of the complement $\Gf \setminus \pi$ containing $v'$.
Then $T$ is a subtree having with the path $\pi$ a unique common vertex, say $v_i$.
Let $e'=(v_i,v'_i)$ be a unique edge belonging to $T$.
Then by the construction $e'$ is of type \nullEdge, whence by~\ref{enum:prof_ab_edges:null} of Lemma~\ref{lm:prof_ab_edges} all other edges of $T$ are also of type \nullEdge.
In particular, so are all edges incident to $v'$, whence $k=0$.
\qed

\section{Diffeomorphisms of non-orientable manifolds}\label{sect:diff_non_or}
Let $\NSurf$ be a smooth non-orientable connected manifold of dimension $\dimM$, $p\colon \OSurf\to\NSurf$ be the oriented double covering of $\NSurf$, and $\xi\colon \OSurf\to \OSurf$ be the corresponding $\Cinfty$ diffeomorphism without fixed points generating the group $\bZ_2$ of covering transformations, that is $\xi^2=\id_{\OSurf}$ and $p\circ\xi= p$.

A diffeomorphism $\tdif\in\Diff(\OSurf)$ will be called \myemph{symmetric} if it commutes with $\xi$, that is $\tdif\circ\xi=\xi\circ\tdif$.
Denote by $\tDiff(\OSurf,\OSubMan)$ the group of all symmetric diffeomorphisms of $\OSurf$ fixed on a closed subset $\OSubMan\subset\OSurf$ and by $\tDiffId(\OSurf,\OSubMan)$ the identity path component of $\tDiff(\OSurf,\OSubMan)$.
If $\OSubMan$ is empty, we will just omit it from notation.

The aim of this section is to find precise relations between the groups $\Diff(\NSurf)$ and $\tDiff(\OSurf)$, see Lemma~\ref{lm:DM_tDtM_relations} below.

\begin{lemma}\label{lm:cov:path_comp_of_preimage}
Let $\NSubMan \subset \Nman$ be a path connected subset.
Then its preimage $\OSubMan=p^{-1}(\NSubMan)$ is either path connected or consists of two disjoint path components which are interchanged by $\xi$.
\end{lemma}
\begin{proof}
One easily deduces from path lifting axiom for the covering map $p:\OSurf\to\NSurf$, that $p(\OSubMan)=\NSubMan$ for every path component $\OSubMan'$ of $\OSubMan$.
Hence for every point $\nx\in\NSubMan$ its inverse image $p^{-1}(\nx)$ intersects each path component of $\OSubMan$.
But $p^{-1}(\nx)$ consists of two points, say $\ox$ and $\xi(\ox)$, whence $\OSubMan$ must consist of either one or two path components.
Moreover, if $\OSubMan$ has two path components $\OSubMan'$ and $\OSubMan''$ such that $\ox\in\OSubMan'$ and $\xi(\ox)\in\OSubMan''$, then $\xi$ interchanges $\ox$ and $\xi(\ox)$ as well as path components $\OSubMan'$ and $\OSubMan''$.
\end{proof}

\begin{lemma}\label{lm:DM_tDtM_relations}
Each $\odif\in \tDiff(\OSurf)$ yields a diffeomorphism $\ndif\in\Diff(\NSurf)$ such that $p\circ\odif = \ndif\circ p$, and the correspondence $\odif\mapsto\ndif$ is a continuous epimorphism $\rho: \tDiff(\OSurf)\to \Diff(\NSurf)$ with kernel $\ker(\rho) = \{\id_{\OSurf}, \xi\} \cong \bZ_2$.
Moreover, $\rho$ yields an isomorphism of $\tDiffPl(\OSurf)$ onto $\Diff(\NSurf)$, so we get the following commutative diagram whose rows are exact and all vertical arrows are isomorphisms:
\begin{equation}\label{equ:tDtM_z2_tDpltM}
\begin{aligned}
\xymatrix{
\ \bZ_2 \ \ar@{^{(}->}[rr]^-{i \,\mapsto\, (i,\id_{\OSurf})} \ar[d]^-{i \,\mapsto\, \xi^i}_-{\cong}  &&
\ \bZ_2\times \tDiffPl(\OSurf) \ \ar@{->>}[rr]^-{(i,\odif) \,\mapsto\, \odif}
  \ar[d]^-{(i,\odif) \,\mapsto\, \xi^i \circ \odif}_-{\cong}  &&
\ \tDiffPl(\OSurf) \ \ar@/^3pt/[d]^-{\rho} \\
\ \langle \xi \rangle \ \ar@{^{(}->}[rr] &&
\ \tDiff(\OSurf) \ \ar@{->>}[rr]^-{\rho} &&
\ \Diff(\NSurf) \ar@/^3pt/[u]^-{s} \
}
\end{aligned}
\end{equation}
where $s$ is the inverse to $\rho$.

Moreover, $s$ also induces the isomorphisms described below.

\begin{enumerate}[leftmargin=*, label={\rm(\arabic*)}]
\item\label{enum:s_iso:DMX}
For every subset $\NSubMan \subset \NSurf$ we have an isomorphism
\begin{equation}\label{equ:s_DidMX_tDidtMtX}
s\colon \DiffId(\NSurf,\NSubMan) \cong \tDiffId(\OSurf, p^{-1}(\NSubMan)).
\end{equation}

\item\label{enum:s_iso:DidMX}
Suppose $\NSubMan \subset \NSurf$ is a subset such that for every path component $\NSubMan'$ of $\NSubMan$ and $\odif\in\tDiffPl(\OSurf)$ the restrictions
\[
\odif|_{p^{-1}(\NSubMan')},\, \xi|_{p^{-1}(\NSubMan')}\colon \, p^{-1}(\NSubMan') \to \OSurf
\]
are \myemph{distinct maps}, that is they take distinct values at some point.
Then we also have an isomorphism
\begin{equation}\label{equ:s_DMX_tDtMtX}
  s\colon \Diff(\NSurf,\NSubMan) \cong \tDiff(\OSurf, p^{-1}(\NSubMan)).
\end{equation}
For instance, this hold if $\NSubMan$ is an $\dimM$-dimensional submanifold or a \myemph{two-sided} $(\dimM-1)$-dimensional submanifold, but does not hold when $\NSubMan$ is a finite subset.
\end{enumerate}
\end{lemma}
\begin{proof}
Let $\odif\in\tDiff(\OSurf)$ and $\ndif = \rho(\odif) \in\Diff(\NSurf)$.
Then $\rho^{-1}(\ndif)$ consists of two diffeomorphisms $\odif$ and $\xi\circ\odif$ one of which preserves orientation, and another one reverses it.
Denote by $s(\ndif)$ those one which preserves orientation.
Then the correspondence $\ndif\mapsto s(\ndif)$ is a continuous homomorphism $s:\Diff(\NSurf) \to \tDiffPl(\OSurf)$ satisfying $\rho\circ s = \id_{\Diff(\NSurf)}$.
Since by definition $\xi$ commutes with all $\tDiff(\OSurf)$ and generates the kernel of $\rho$, we get the desired diagram~\eqref{equ:tDtM_z2_tDpltM}.

\ref{enum:s_iso:DMX}
First notice that $\tDiffId(\OSurf)$ is also the identity path component of $\tDiffPl(\OSurf)$.
Hence $\rho$ induces an isomorphism of $\tDiffId(\OSurf)$ onto the path component $\DiffId(\NSurf)$ of $\Diff(\NSurf)$, and so we get the inverse isomorphism $s\colon \DiffId(\NSurf) \cong \tDiffId(\OSurf)$ coinciding with~\eqref{equ:s_DidMX_tDidtMtX} for the case $\NSubMan=\varnothing$.

Suppose now that $\NSubMan \subset \NSurf$ is a non-empty subset and let $\OSubMan = p^{-1}(\NSubMan)$.
Evidently, $\rho\bigl( \tDiff(\OSurf,\OSubMan) \bigr) \subset \Diff(\NSurf,\NSubMan)$, that is if $\ndif\in\tDiff(\OSurf)$ is fixed on $\OSubMan$, then $\rho(\odif)$ is fixed on $\NSubMan$.
Hence
\[ \rho\bigl( \tDiffId(\OSurf,\OSubMan) \bigr) \subset \DiffId(\NSurf,\NSubMan). \]

Conversely, let $\ndif\in \DiffId(\NSurf,\NSubMan)$, so there is an isotopy $H:\NSurf\times [0,1]\to\NSurf$ such that $H_0=\id_{\NSurf}$, $H_1=\ndif$, and each $H_t$ is fixed on $\NSubMan$.
Since $p$ is a covering map, $H$ lifts to a unique isotopy $\cov{H}:\OSurf\times[0,1]\to\OSurf$ such that $\cov{H}_0=\id_{\OSurf}$ and $\rho(\cov{H}_t)=H_t$.
In particular, $\cov{H}_t\in\tDiffId(\OSurf) \subset \tDiffPl(\OSurf)$, and so $\cov{H}_t = s(H_t)$.

It remains to show that each $\cov{H}_t$ is fixed on $\OSubMan$, which will imply that
\[
   s(\DiffId(\NSurf,\NSubMan)) \subset \tDiffId(\OSurf,\OSubMan)
\]
and give the isomorphism~\ref{enum:s_iso:DMX}.
Let $\ox\in\OSubMan$ and $\nx=p(\ox) \in \NSubMan$.
Since $H(x\times [0,1]) = \nx$, it follows that $\cov{H}(\ox \times [0,1]) \subset p^{-1}(\nx) = \{\ox, \xi(\ox)\}$.
But the latter set is discrete and $\cov{H}(\ox,0) = \ox$, whence $\cov{H}(\ox \times [0,1])=\ox$ as well.
Thus $\cov{H}_t$ is fixed on $\OSubMan$.

\medskip

\ref{enum:s_iso:DidMX}
Let $\OSubMan = p^{-1}(\NSubMan)$, so the restriction $p:\OSubMan \to \NSubMan$ is a double covering map.
As noted above $\rho\bigl( \tDiff(\OSurf,\OSubMan) \bigr) \subset \Diff(\NSurf,\NSubMan)$, and so we should only check that
\begin{equation}\label{equ:s_mX_tDtMtX}
s\bigl(\Diff(\NSurf,\NSubMan)\bigr) \subset \tDiff(\OSurf,\OSubMan).
\end{equation}

a) Suppose $\NSubMan$ is path connected and let $\ndif\in \Diff(\NSurf,\NSubMan)$ and $\odif = s(\ndif) \in \tDiffPl(\OSurf)$.
To prove~\eqref{equ:s_mX_tDtMtX} we should check that $\odif$ is fixed on $\OSubMan$.

Since $\ndif$ is fixed on $\NSubMan$, it follows that $\odif(\ox) \in \{ \ox, \xi(\ox)\}$ for all $\ox\in\OSubMan$.
By assumption $\ndif|_{\OSubMan} \not= \xi|_{\OSubMan}$, so there exists a point $\ox \in\OSubMan$ such that $\odif(\ox)\not=\xi(\ox)$, and so $\odif(\ox) = \ox$.
Let $\OSubMan'$ be the path component of $\OSubMan$ containing $\ox$.
Then $\odif(\OSubMan')=\OSubMan'$ and the restriction $\odif|_{\OSubMan'}:\OSubMan' \to \OSubMan'$ is a unique lifting of the identity map $\id_{\NSubMan}:\NSubMan\to\NSubMan$ for the covering map $p|_{\OSubMan}: \OSubMan \to \NSubMan$ having the property that $\odif(\ox)= \ox$.
Hence $\odif|_{\OSubMan}$ is the identity, i.e.\! $\odif$ is fixed on $\OSubMan'$.

Furthermore, suppose there exists another path component $\OSubMan''$ of $\OSubMan$.
Then by Lemma~\ref{lm:cov:path_comp_of_preimage} $\xi(\OSubMan') = \OSubMan''$ and $\xi(\ox)\in\OSubMan''$.
Since $\odif(\OSubMan')=\OSubMan'$, it follows that $\odif(\OSubMan'')=\OSubMan''$ and therefore $\odif(\xi(\ox)) = \xi(\ox)$.
Hence $\odif$ has a fixed point in $\OSubMan''$, and so it is fixed on $\OSubMan''$ as well.
In other words $s(\ndif) = \odif \in \tDiff(\OSurf,\OSubMan)$, which proves~\eqref{equ:s_mX_tDtMtX} for the case when $\NSubMan$ is path connected.

\medskip

b) Now suppose $\NSubMan$ is not path connected, and let $\{ \NSubMan_i\}_{i\in\Lambda}$ the collection of all path components of $\NSubMan$, so $\NSubMan = \mathop{\cup}\limits_{i\in\Lambda}\NSubMan_i$.
Then by a)
\[ s\bigl(\Diff(\NSurf,\NSubMan_{i})\bigr) = \tDiff(\OSurf,p^{-1}(\NSubMan_i)), \qquad i\in\Lambda.\]
Hence
\begin{align*}
s\bigl(\Diff(\NSurf,\NSubMan)\bigr) &=
s\bigl(\Diff(\NSurf,\mathop{\cup}\limits_{i\in\Lambda}\NSubMan_i)\bigr) =
s\bigl(\mathop{\cap}\limits_{i\in\Lambda} \Diff(\NSurf,\NSubMan_i)\bigr) =
\mathop{\cap}\limits_{i\in\Lambda} s\bigl( \Diff(\OSurf,\NSubMan_i)\bigr) = \\
&= \mathop{\cap}\limits_{i\in\Lambda} \tDiff(\OSurf,p^{-1}(\NSubMan_i)\bigr) =
 \tDiff(\OSurf, \mathop{\cup}\limits_{i\in\Lambda} p^{-1}(\NSubMan_i)\bigr) =
 \tDiff(\OSurf, p^{-1}(\NSubMan)\bigr).
\end{align*}
Lemma is proved.
\end{proof}

\begin{lemma}\label{lm:SttSt}
Let $\nfunc:\NSurf\to\Pman$ be a $\Cinfty$, $\ofunc = \nfunc\circ p:\OSurf\to\Pman$,
\begin{align*}
\Stabilizer{\nfunc} &:= \{ \ndif\in\Diff(\NSurf) \mid \nfunc\circ\ndif = \nfunc\}, &
\tStabilizer{\ofunc} &:= \{ \odif\in\tDiff(\NSurf) \mid \ofunc\circ\odif = \ofunc\}.
\end{align*}
The following statements hold true.
\begin{enumerate}[label={\rm(\alph*)}, itemsep=0.6ex]
\item\label{enum:Stab_tStab}
 $\rho(\tStabilizer{\ofunc)} = \Stabilizer{\nfunc}$ and $\rho^{-1}(\Stabilizer{\nfunc}) = \tStabilizer{\ofunc}$;
\item\label{enum:SfX_tStftX}
Suppose $\dim\NSurf=2$, $\nfunc\in\FSp{\NSurf}$, and let $\NSubMan$ be an $\nfunc$-adopted submanifold.
Then $s$ induces an isomorphism
\begin{equation}\label{equ:s_iso_Stabs}
s: \Stabilizer{\nfunc,\NSubMan} \cong\tStabilizer{\ofunc,p^{-1}(\NSubMan)}.
\end{equation}
\end{enumerate}
\end{lemma}
\begin{proof}
Let $\odif\in\tDiff(\OSurf)$ and $\ndif=\rho(\odif)\in\Diff(\NSurf)$, so
\begin{align}\label{equ:h_th_diffs}
  \odif\circ\xi&=\xi\circ \odif, &
  p\circ \odif&=\ndif\circ p, &
  \ofunc&=\nfunc\circ p.
\end{align}
We have to show that $\ndif\in\Stabilizer{\nfunc}$ if and only if $\odif\in\tStabilizer{\ofunc}$, i.e.\! we need to deduce from~\eqref{equ:h_th_diffs} an equivalence of the following relations:
\begin{align*}
  \nfunc\circ\ndif&=\nfunc, &
  \ofunc\circ\odif&=\ofunc.
\end{align*}
Let $\ox\in\OSurf$ and $\nx=p(\ox)$.
If $\ofunc\circ\odif(\ox)=\ofunc(\ox)$, then
\[
  \nfunc\circ\ndif(\nx) =
  \nfunc\circ\ndif\circ p(\nx) =
  \nfunc\circ p \circ \odif(\ox)  =
  \ofunc \circ \odif(\ox)  =
  \ofunc(\ox)  =
  \nfunc \circ p(\ox) =\nfunc(\nx).
\]
Conversely, if $\nfunc\circ\ndif(\nx)=\nfunc(\nx)$, then
\[
  \ofunc\circ\odif(\ox) =
  \nfunc\circ p \circ \odif(\ox) =
  \nfunc\circ \ndif \circ p(\ox) =
  \nfunc \circ p(\ox) =
  \ofunc(\ox).
\]

\medskip

\ref{enum:SfX_tStftX}
Denote $\OSubMan = p^{-1}(\NSubMan)$.
Since $\OSubMan$ is an $\ofunc$-adopted submanifold, one easily checks that $\Stabilizer{\ofunc,\OSubMan} \subset \DiffPl(\OSurf)$.
Hence by~\ref{enum:Stab_tStab} and Lemma~\ref{lm:DM_tDtM_relations} $\rho$ injectively maps $\tStabilizer{\ofunc,\OSubMan}$ into $\Stabilizer{\nfunc,\NSubMan}$.

Conversely, it follows from~\ref{enum:Stab_tStab} that $s(\Stabilizer{\nfunc})\subset \tStabilizer{\ofunc}$.
Therefore we get from statement~\ref{enum:s_iso:DidMX} of Lemma~\ref{lm:DM_tDtM_relations} that
\begin{align*}
s\bigl( \Stabilizer{\nfunc,\NSubMan} \bigr) &=
s\bigl( \Stabilizer{\nfunc}\cap\Diff(\NSurf,\NSubMan) \bigr) \subset
\tStabilizer{\nfunc} \cap \tDiff(\OSurf,\OSubMan) =
\tStabilizer{\nfunc,\OSubMan}.
\end{align*}
Since $\rho\circ s = \id_{\DiffPl(\OSurf)}$, it follows that $s$ isomorphically maps $\Stabilizer{\nfunc,\NSubMan}$ onto $\tStabilizer{\nfunc,\OSubMan}$.
\end{proof}

\section{Hamiltonian like flows for $\ofunc\in\FSp{\OSurf}$.}\label{sect:ham_vh}
Let $\OSurf$ be an orientable compact surface.
\begin{definition}\label{def:ham_like_vf}
Let $\ofunc\in\FSp{\OSurf}$ and $\fSing$ be the set of critical points of $\ofunc$.
A smooth vector field $\fld$ on $\OSurf$ will be called \myemph{Hamiltonian like} for $\ofunc$ if the following conditions hold true.
\begin{enumerate}[leftmargin=5ex, topsep=0pt, label={\rm(\alph*)}]
\item\label{enum:hf:crpt}
$\fld(z)=0$ if and only if $z$ is a critical point of $\ofunc$.
\item\label{enum:hf:pres_func}
$\fld(\ofunc)\equiv0$ everywhere on $\OSurf$, that is $\ofunc$ is constant along orbits of $\fld$.
\item\label{enum:hf:local_form}
Let $z$ be a critical point of $\ofunc$.
Then there exists a local representation of $\ofunc$ at $z$ as a homogeneous polynomial $v:(\bR^2,0)\to(\bR,0)$ without multiple factors (as in Definition~\ref{def:classF}) such that in the same coordinates $(x,y)$ near the origin $0$ in $\bR^2$ we have that $\fld = -v'_{y}\,\tfrac{\partial}{\partial x} + v'_{x}\,\tfrac{\partial}{\partial y}$.
\end{enumerate}
\end{definition}
It follows from~\ref{enum:hf:crpt} and Definition~\ref{def:classF} that every orbit of $\fld$ is of one of the following types:
\begin{itemize}
 \item a critical point of $\ofunc$;
 \item a regular component of some level set of $\ofunc$, and so it is a closed orbit of $\fld$;
 \item a connected component of the sets $\crLev\setminus\fSing$, where $\crLev$ runs over all critical components of level-sets of $\ofunc$.
\end{itemize}
By \cite[Lemma~5.1]{Maksymenko:AGAG:2006} or \cite[Lemma~16]{Maksymenko:ProcIM:ENG:2010} for every $\ofunc\in\FSp{\OSurf}$ there exists a Hamiltonian like vector field.
For the proof take the Hamiltonian vector field $\fld$ for $\ofunc$ with respect to any symplectic form $\orb$ on $\OSurf$, and then properly change $\fld$ near each critical point of $\ofunc$ in accordance with~\ref{enum:hf:local_form} of Definition~\ref{def:ham_like_vf}.

Let $\fld$ be a Hamiltonian like vector field for $\ofunc$.
Since $\ofunc$ takes constant values on boundary components of $\OSurf$, it follows that $\fld$ is tangent to $\partial\OSurf$ and therefore it generates a flow $\flow:\OSurf \times \bR\to \OSurf$ which will also be called \myemph{Hamiltonian like} for $\ofunc$.

For each smooth function $\alpha:\OSurf\to\bR$ let $\flow_{\alpha}:\OSurf\to\OSurf$ be the map defined by
\begin{equation}\label{equ:shift_via_alpha}
\flow_{\alpha}(x) = \flow(x,\alpha(x)), \qquad x\in\OSurf.
\end{equation}
We will call $\flow_{\alpha}$ the \myemph{shift along orbits of $\flow$} via the function $\alpha$, and $\alpha$ will in turn be called a \myemph{shift function} for $\flow_{\alpha}$.

Evidently, condition~\ref{enum:hf:pres_func} of Definition~\ref{def:ham_like_vf} is equivalent to the assumption that
\[\ofunc\circ\flow_t = \ofunc\] for all $t\in\bR$, that is $\flow_t\in\Stabilizer{\ofunc}$.

More generally, since $\flow_{\alpha}$ leaves invariant each orbit of $\flow$, we see that $\ofunc\circ\flow_{\alpha}=\ofunc$ for every function $\alpha\in\Ci{\OSurf}{\bR}$.
In particular, $\flow_{\alpha}$ is a \myemph{diffeomorphism} if and only if $\flow_{\alpha} \in \Stabilizer{\ofunc}$.
Moreover, in this case $\flow_{\alpha} \in \StabilizerId{\ofunc}$ and $\{\flow_{t\alpha}\}_{t\in[0,1]}$ is an isotopy between $\id_{\OSurf}$ and $\flow_{\alpha}$.

Denote by $\fld(\alpha)$ the Lie derivative of $\alpha$ with respect to $\fld$ and let
\begin{equation}\label{equ:shift_funcs_for_diff}
    \DFunc = \{ \alpha\in\Ci{\OSurf}{\bR} \mid 1+\fld(\alpha)>-0 \}.
\end{equation}

\begin{theorem}\label{th:charact_Stabf}{\rm\cite[Theorem~1.3]{Maksymenko:AGAG:2006}, \cite[Theorem~3]{Maksymenko:ProcIM:ENG:2010}.}
Let $\ofunc\in\FSp{\OSurf}$, $\flow:\OSurf\times\bR\to\OSurf$ be the flow generated by some Hamiltonian vector field $\fld$, and $\Sh{\fld}:\DFunc \to \StabilizerId{\ofunc}$ be the map defined by $\Sh{\fld}(\alpha) = \flow_{\alpha}$.
If $\ofunc$ has at least one \myemph{saddle} or a \myemph{degenerate local extreme}, then $\Sh{\fld}$ is a \myemph{homeomorphism} with respect to $C^{\infty}$ topologies and $\StabilizerId{\ofunc}$ is contractible (because $\DFunc$ is convex).

Otherwise, there exists $\theta\in\DFunc$ such that $\Sh{\fld}$ can be represented as a composition
\[
\Sh{\fld}: \DFunc \xrightarrow{~\text{quotient}~}
\DFunc/\{n\theta\}_{n\in\bZ} \xrightarrow{~\text{homeomorphism}~}
\StabilizerId{\ofunc}
\]
of the quotient map by the closed discrete subgroup $Z = \{n\theta\}_{n\in\bZ}$ of $\DFunc$ and a homeomorphism of the quotient of $\DFunc$ by $Z$ onto $\StabilizerId{\ofunc}$.
In particular, $\Sh{\fld}$ is an infinite cyclic covering map and $\StabilizerId{\ofunc}$ is homotopy equivalent to the circle.
\end{theorem}


\section{Groups $\FolStabilizer{\nfunc}$}\label{sect:FolStab_Grp}
For $\nfunc\in\FSp{\NSurf}$ let $\FolStabilizer{\nfunc}$ be the \myemph{normal} subgroup of $\Stabilizer{\nfunc}$ consisting of diffeomorphisms $\ndif$ of $\NSurf$ having the following two properties:
\begin{enumerate}[topsep=0pt]
\item[1)]
$\ndif$ leaves invariant every connected component of each level-set of $\nfunc$;
\item[2)]
if $z$ is a \myemph{degenerate local extreme} of $\nfunc$, so, in particular, $\ndif(z)=z$, then the tangent map $T_{z}\ndif: T_z\NSurf\to T_z\NSurf$ is the identity.
\end{enumerate}
For a closed subset $\NSubMan$ of $\NSurf$ define the following three groups:
\begin{align*}
\Diff(\NSurf,\NSubMan) &:= \{\ndif\in\Diff(\NSurf) \mid \ndif \ \text{is fixed on} \ \NSubMan \}, \\
\DiffNbh(\NSurf,\NSubMan) &:= \{\ndif\in\Diff(\NSurf) \mid \ndif \ \text{is fixed on some neighborhood of} \ \NSubMan \}, \\
\DiffId(\NSurf,\NSubMan) &:= \{\ndif\in\Diff(\NSurf,\NSubMan) \mid \ndif \ \text{is isotopic to $\id_{\NSurf}$ rel. $\NSubMan$} \}.
\end{align*}
Define also their intersections with $\FolStabilizer{\nfunc}$ and $\Stabilizer{\nfunc}$ as follows:
\begin{equation}\label{equ:subgroups_of_fStabX}
\begin{aligned}
\FolStabilizer{\nfunc,\NSubMan}&:= \FolStabilizer{\nfunc} \cap \Diff(\NSurf,\NSubMan), & \ \qquad\
\Stabilizer{\nfunc,\NSubMan} &:= \Stabilizer{\nfunc}\cap \Diff(\NSurf,\NSubMan), \\
\FolStabilizerNbh{\nfunc,\NSubMan}&:= \FolStabilizer{\nfunc} \cap \DiffNbh(\NSurf,\NSubMan), & \
\StabilizerNbh{\nfunc,\NSubMan} &:= \Stabilizer{\nfunc}\cap \DiffNbh(\NSurf,\NSubMan), \\
\FolStabilizerIsotId{\nfunc,\NSubMan}&:= \FolStabilizer{\nfunc} \cap \DiffId(\NSurf,\NSubMan), & \
\StabilizerIsotId{\nfunc,\NSubMan} &:= \Stabilizer{\nfunc}\cap \DiffId(\NSurf,\NSubMan), \\
\FolStabilizerNbhIsotId{\nfunc,\NSubMan}&:= \FolStabilizerIsotId{\nfunc} \cap \DiffNbh(\NSurf,\NSubMan), & \
\StabilizerNbhIsotId{\nfunc,\NSubMan} &:= \StabilizerIsotId{\nfunc}\cap \DiffNbh(\NSurf,\NSubMan),
\end{aligned}
\end{equation}
where we follow the notation convention that $\NSubMan$ is omitted if it is empty.
In particular, $\FolStabilizerIsotId{\nfunc}=\FolStabilizer{\nfunc} \cap \DiffIdM$.
The following lemma can be proved similarly to \cite[Lemma~3.4]{Maksymenko:DefFuncI:2014} and we leave its proof for the reader.
\begin{lemma}\label{lm:Lf_Sf}{\rm cf.~\cite[Lemma~3.4]{Maksymenko:DefFuncI:2014}.}
All the groups in~\eqref{equ:subgroups_of_fStabX} are normal subgroups of $\Stabilizer{\nfunc,\NSubMan}$.

The groups $\FolStabilizerIsotId{\nfunc,\NSubMan}$, $\FolStabilizer{\nfunc,\NSubMan}$, $\StabilizerIsotId{\nfunc,\NSubMan}$ are unions of path components of $\Stabilizer{\nfunc,\NSubMan}$.
In particular, $\StabilizerId{\nfunc,\NSubMan}$ is the identity path component of each of these groups.

Similarly, the groups $\FolStabilizerNbhIsotId{\nfunc,\NSubMan}$, $\FolStabilizerNbh{\nfunc,\NSubMan}$, and $\StabilizerNbhIsotId{\nfunc,\NSubMan}$ are also unions of path components of $\StabilizerNbh{\nfunc,\NSubMan}$.
\qed
\end{lemma}

It follows that $\pi_0\FolStabilizer{\nfunc,\NSubMan}$ can be regarded as a normal subgroup of $\pi_0\Stabilizer{\nfunc,\NSubMan}$.
Moreover, if $\nfunc$ \myemph{has no degenerate local extremes}, then the corresponding quotient
\begin{align*}
\GrpKR{\nfunc,\NSubMan}
 &:=\dfrac{\Stabilizer{\nfunc,\NSubMan}}{\FolStabilizer{\nfunc,\NSubMan}} =
\Bigl. \frac{\Stabilizer{\nfunc,\NSubMan}}{\StabilizerId{\nfunc,\NSubMan}} \Bigr/
\frac{\FolStabilizer{\nfunc,\NSubMan}}{\StabilizerId{\nfunc,\NSubMan}}
= \dfrac{\pi_0\Stabilizer{\nfunc,\NSubMan}}{\pi_0\FolStabilizer{\nfunc,\NSubMan}}.
\end{align*}
can be interpreted as the group of automorphisms of the Kronrod-Reeb graph of $\nfunc$ induced by diffeomorphisms from $\Stabilizer{\nfunc,\NSubMan}$, see e.g.~\cite{Kronrod:UMN:1950}, \cite{Reeb:ASI:1952}, \cite{BolsinovFomenko:ENG:2004}, \cite{Maksymenko:ProcIM:ENG:2010}.
If $\nfunc$ has degenerate local extremes, then there is a similar interpretation of $\GrpKR{\nfunc,\NSubMan}$ but one should modify Kronrod-Reeb graph of $\nfunc$ by gluing additional edges to each vertex corresponding to each degenerate local extreme, see for details~\cite{Maksymenko:ProcIM:ENG:2010}.
Similarly, one can define
\begin{align*}
\GrpKRIsotId{\nfunc,\NSubMan} &=
 \dfrac{\pi_0\StabilizerIsotId{\nfunc,\NSubMan}}{\pi_0\FolStabilizerIsotId{\nfunc,\NSubMan}}, &
\GrpKRNbh{\nfunc,\NSubMan} &=
 \dfrac{\pi_0\StabilizerNbh{\nfunc,\NSubMan}}{\pi_0\FolStabilizerNbh{\nfunc,\NSubMan}}, &
\GrpKRNbhIsotId{\nfunc,\NSubMan} &=
 \dfrac{\pi_0\StabilizerNbhIsotId{\nfunc,\NSubMan}}{\pi_0\FolStabilizerNbhIsotId{\nfunc,\NSubMan}}.
\end{align*}

Our aim is to prove the following statement extending~\cite[Corollary~7.2]{Maksymenko:DefFuncI:2014} to non-orientable case and deduce from it several useful results.

\begin{theorem}\label{th:hom_equ}
{\rm cf. \cite[Corollary~6.1]{Maksymenko:AGAG:2006}, \cite[Corollary~7.2]{Maksymenko:DefFuncI:2014}.}
Let $\NSurf$ be a compact surface, $\nfunc\in\FSp{\NSurf}$, $\NSubMan\subset\NSurf$ be a compact $\nfunc$-adopted submanifold, and $\NX$ be an $\nfunc$-regular neighborhood of $\NSubMan$.
Then the following inclusions are homotopy equivalences:
\begin{eqnarray}
\label{equ:SfNX_SfnbX_SfX}
\Stabilizer{\nfunc,\NX} \ \subset \
\StabilizerNbh{\nfunc,\NSubMan} \ \subset \
\Stabilizer{\nfunc,\NSubMan}, \\
\label{equ:SprfNX_SprfnbX_SprfX}
\StabilizerIsotId{\nfunc,\NX} \ \subset \
\StabilizerNbhIsotId{\nfunc,\NSubMan} \ \subset \
\StabilizerIsotId{\nfunc,\NSubMan}, \\
\label{equ:DeltafNX_DeltafnbX_DeltafX}
\FolStabilizer{\nfunc,\NX} \ \subset \
\FolStabilizerNbh{\nfunc,\NSubMan} \ \subset \
\FolStabilizer{\nfunc,\NSubMan}, \\
\label{equ:DeltaPrfNX_DeltaPrfnbX_DeltaPrfX}
\FolStabilizerIsotId{\nfunc,\NX} \ \subset \
\FolStabilizerNbhIsotId{\nfunc,\NSubMan} \ \subset \
\FolStabilizerIsotId{\nfunc,\NSubMan}.
\end{eqnarray}
\end{theorem}
\begin{proof}
The case when $\NSurf$ is orientable is proved in~\cite{Maksymenko:DefFuncI:2014}.
So our aim is to extend it to the case when $\NSurf$ is non-orientable.
In fact the proof is an adaptation of~\cite[Lemma~7.1]{Maksymenko:DefFuncI:2014} similar to~\cite[Lemma~4.14]{Maksymenko:AGAG:2006} and therefore we only indicate the principal arguments.

\medskip

Also notice that similarly to~\cite[Corollary~7.2]{Maksymenko:DefFuncI:2014} it suffices to prove that the inclusions~\eqref{equ:SfNX_SfnbX_SfX} are homotopy equivalences, so they induce bijections between the path components of the corresponding groups, and the inclusions of the corresponding path components are homotopy equivalences.

Indeed, notice that the groups in~\eqref{equ:SprfNX_SprfnbX_SprfX} are intersections of the corresponding groups from~\eqref{equ:SfNX_SfnbX_SfX} with the path component $\DiffId(\NSurf)$ of the larger group $\Diff(\NSurf)$.
If a path component $\mathcal{K}$ of any group in~\eqref{equ:SprfNX_SprfnbX_SprfX} intersects $\DiffId(\NSurf)$, then $\mathcal{K}$ is contained in $\DiffId(\NSurf)$.
Hence the inclusions~\eqref{equ:SprfNX_SprfnbX_SprfX} yield bijections between the path components of the corresponding groups, and due to~\eqref{equ:SfNX_SfnbX_SfX} the inclusions of path components are homotopy equivalences.

The deduction that~\eqref{equ:DeltafNX_DeltafnbX_DeltafX} and~\eqref{equ:DeltaPrfNX_DeltaPrfnbX_DeltaPrfX} are homotopy equivalences is similar.

\medskip

Thus assume that $\NSurf$ is a non-orientable connected compact surface.
Consider its oriented double covering $p\colon \OSurf\to\NSurf$, and let $\xi\colon \OSurf\to \OSurf$ be the corresponding $\Cinfty$ diffeomorphism without fixed points generating the group $\bZ_2$ of covering transformations, that is $\xi^2=\id_{\OSurf}$ and $p\circ\xi= p$.

Denote $\ofunc = \nfunc\circ p: \OSurf\to\Pman$, $\OSubMan = p^{-1}(\NSubMan)$, and $\NtX = p^{-1}(\NX)$.
Then $\ofunc\in\FSp{\OSurf}$, $\NSubMan\subset\NSurf$ is a compact $\ofunc$-adopted submanifold of $\OSurf$, and $\NtX$ is a $\ofunc$-regular neighborhood of $\OSubMan$.

Fix a Hamiltonian like vector field $\fld$ for $\ofunc$ on $\OSurf$ and let $\flow:\OSurf\times\bR\to\OSurf$ be the flow generated by $\fld$.

Let $\xi^{*}\fld = T\xi^{-1} \circ \fld\circ \xi:\OSurf \to T\OSurf$ be the vector field on $\OSurf$ induced by $\xi$ from $\fld$.
Then one can always assume, see~\cite[Lemma~5.1 (2)]{Maksymenko:AGAG:2006}, that $\fld$ is also \myemph{skew-symmetric} with respect to $\xi$ in the sense that $\xi^{*}\fld = -\fld$, whence
\begin{equation}\label{equ:skew_symm_flow}
\xi \circ \flow_t = \flow_{-t} \circ \xi
\end{equation}
for all $t\in\bR$.
Indeed, it is necessary to replace $\fld$ with $\frac{1}{2}(\fld + \xi^*\fld)$ and properly change it near critical points of $\func$ in order to preserve property~\ref{enum:hf:local_form} of Definition~\ref{def:ham_like_vf}.

\begin{sublemma}\label{lm:deform_in_stab}
{\rm cf. \cite[Lemma~4.14]{Maksymenko:AGAG:2006}, \cite[Lemma~7.1]{Maksymenko:DefFuncI:2014}.}
Let $\mathcal{A}\subset\Stabilizer{\ofunc}$ be a subset and $\gamma\colon\mathcal{A}\to \Cinfty(\OSubMan,\bR)$ be a continuous map such that
\begin{equation}\label{equ:gamma_shift_func}
\odif(x) = \flow(x, \gamma(\odif)(x))
\end{equation}
for all $\odif\in\mathcal{A}$ and $x\in\OSubMan$.
Then for any pair $\Uman\subset\OSubSurf$ of $\ofunc$-regular neighborhoods of $\OSubMan$ such that $\overline{\Uman}\subset \Int{\OSubSurf}$, there exists a continuous map $\beta\colon\mathcal{A}\to \DFunc\subset \Cinfty(\OSurf,\bR)$, see~\eqref{equ:shift_funcs_for_diff}, satisfying the following conditions.
\begin{enumerate}[leftmargin=*, wide, label={\rm(\arabic*)}, itemsep=0.6ex]
\item\label{enum:deform_in_stab:beta_ext_gamma}
For each $\odif\in\mathcal{A}$ the function $\beta(\odif)$ extends $\gamma(\odif)$ to all $\OSurf$, satisfies relation~\eqref{equ:gamma_shift_func} on $\Uman$, and vanishes on $\overline{\OSurf\setminus\OSubSurf}$, that is
\begin{itemize}[leftmargin=8ex, itemsep=0.8ex]
    \item $\beta(\odif) = \gamma(\odif)$ on $\OSubMan$,
    \item $\odif(x)=\flow(x, \beta(\odif)(x))$ for all $x\in\Uman$,
    \item $\beta(\odif)=0$ on $\overline{\OSurf\setminus\OSubSurf}$.
\end{itemize}

\item\label{enum:deform_in_stab:idW_has_zero_beta}
If $\gamma(\odif)=0$ and $\odif$ is fixed on some $\ofunc$-regular neighborhood $\Uman' \subset \Uman$, then $\beta(\odif)\equiv0$ on $\Uman'$ as well.

\item\label{enum:deform_in_stab:homotopy}
The homotopy $H\colon \mathcal{A} \times I\to\Stabilizer{\ofunc}$ defined by
\[ H(\odif,t)=(\flow_{t\beta(\odif)})^{-1}\circ \odif \]
has the following properties:
\begin{enumerate}[leftmargin=8ex, label={\rm(\alph*)}, itemsep=0.6ex, topsep=1ex]
\item\label{enum:deform_in_stab:homotopy:a}
$H_0=\id_{\mathcal{A}}$ and $H_1(\mathcal{A})\subset\Stabilizer{\ofunc,\Uman}$, so it deforms $\mathcal{A}$ in $\Stabilizer{\ofunc}$ into $\Stabilizer{\ofunc,\Uman}$;

\item\label{enum:deform_in_stab:homotopy:b}
if $\gamma(\odif)\equiv0$ and $\odif$ is fixed on some $\ofunc$-regular neighborhood $\Uman' \subset \Uman$, then $H_t(\odif)$ is fixed on $\Uman'$ for all $t\in[0,1]$ as well.

\end{enumerate}
\smallskip
\end{enumerate}

Suppose in addition that $\fld$ is skew-symmetric, that is $\xi^{*}\fld =-\fld$, and either of the following conditions hold:
\begin{enumerate}[label={\rm(\roman*)}, itemsep=0.6ex]
\item\label{enum:skewsym:saddles}
every connected component of $\OSubMan$ contains a  critical point of $\ofunc$ being not a \myemph{non-degenerate local extreme};

\item\label{enum:skewsym:gamma_is_zero}
$\gamma(\odif)\equiv0$ for all $\odif\in\mathcal{A}\cap\tStabilizer{\ofunc}$.
\end{enumerate}
Then one can assume that
\begin{enumerate}[label={\rm(\arabic*)}, itemsep=1ex, start=4]
\item\label{enum:deform_in_stab:beta_skew_symm}
$\beta(\odif)\circ\xi=-\beta(\odif)$ for each $\odif\in\mathcal{A}\cap\tStabilizer{\ofunc}$;

\item\label{enum:deform_in_stab:skew_symm_Hinv}
$H\bigl( (\mathcal{A}\cap\tStabilizer{\ofunc})\times I \bigr)\subset\tStabilizer{\ofunc}$, that is the set of $\xi$-symmetric diffeomorphisms remains invariant with respect to the homotopy $H$.
\end{enumerate}

\end{sublemma}
\begin{proof}
Statements~\ref{enum:deform_in_stab:beta_ext_gamma}-\ref{enum:deform_in_stab:homotopy} constitute~\cite[Lemma~7.1]{Maksymenko:DefFuncI:2014}.
So we should verify statements~\ref{enum:deform_in_stab:beta_skew_symm} and~\ref{enum:deform_in_stab:skew_symm_Hinv} concerning skew-symmetric diffeomorphisms.

Let us briefly recall the idea of proof.
Since $\NtX$ is a $\ofunc$-regular neighborhood of $\OSubMan$, for each $\odif\in\mathcal{A}$ the function $\gamma(\odif)\colon\OSubMan\to\bR$ uniquely extends to a $\Cinfty$ function $\cov{\gamma}(\odif)\colon\OSubSurf\to\bR$ such that~\eqref{equ:gamma_shift_func} holds on $\OSubSurf$, that is $\odif(x)=\flow(x,\cov{\gamma}(\odif)(x))$ for all $x\in\OSubSurf$.
Moreover, the correspondence $\odif\to \,\cov{\gamma}(\odif)$ is a continuous map $\,\cov{\gamma}\colon\mathcal{A}\to \Cinfty(\OSubSurf,\bR)$.

Fix a $\Cinfty$ function $\mu\colon \OSurf\to [0,1]$ with the following properties:
\begin{itemize}
\item $\mu=0$ on some neighborhood of $\overline{\OSurf\setminus\OSubSurf}$;
\item $\mu=1$ on some neighborhood of $\overline{\Uman}$;
\item $\fld(\mu)=0$, that is $\mu$ take constant values along orbits of $\fld$.
\end{itemize}
Then the required map $\beta\colon\mathcal{A}\to \Cinfty(\OSurf,\bR)$ can be defined by
\begin{equation}\label{equ:beta_map}
\beta(\odif)(x)=
\begin{cases}
\cov{\gamma}(\odif)(x)\cdot \mu(x), &\text{for} \ x\in \OSubSurf,\\
0, & \text{for} \ x\in\OSurf\setminus\OSubSurf.
\end{cases}
\end{equation}

\medskip

Suppose now that $\fld$ is skew-symmetric with respect to $\xi$.
We will show below that in this case
\begin{equation}\label{equ:-g_xi__g}
  -\cov{\gamma}(\odif)\circ\xi = \cov{\gamma}(\odif),
\end{equation}
for all $\odif\in\mathcal{A}\cap\tStabilizer{\ofunc}$.

Assuming that~\eqref{equ:-g_xi__g} holds let us complete the proof of Lemma~\ref{lm:deform_in_stab}.
Since $\Uman$ and $\OSubSurf$ are invariant with respect to $\xi$, and $\xi$ maps orbits of $\fld$ onto orbits, one can replace the function $\mu$ with $\frac{1}{2}(\mu + \mu\circ \xi)$ not violating the above conditions on $\mu$ and thus additionally assume that
\begin{equation} \label{equ:mu_xi__mu}
  \mu\circ \xi = \mu.
\end{equation}
Now if we define $\beta$ by the same formula~\eqref{equ:beta_map} then conditions~\ref{enum:deform_in_stab:beta_skew_symm} and~\ref{enum:deform_in_stab:skew_symm_Hinv} will hold true.

\medskip

\ref{enum:deform_in_stab:beta_skew_symm}
If $\odif\in\mathcal{A}\cap\tStabilizer{\ofunc}$ and $x\in\OSubSurf$, then
\[
\beta(\odif)\circ\xi(x) =
 \cov{\gamma}(\odif)\circ\xi(x) \cdot \mu \circ\xi(x)
 \stackrel{\eqref{equ:-g_xi__g}, \, \eqref{equ:mu_xi__mu}}{=\!=\!=\!=\!=\!=\!=}
-\cov{\gamma}(\odif)(x) \cdot \mu(x) =
-\beta(\odif)(x).
\]
On the other hand, if $x\in\OSurf\setminus\OSubSurf$, then $\xi(x)\in\OSurf\setminus\OSubSurf$ as well, and so
\[
  \beta(\odif)(x) = \beta(\odif)\circ\xi(x) = 0.
\]

\medskip

\ref{enum:deform_in_stab:skew_symm_Hinv}
Notice that for each $\odif\in\mathcal{A}\cap\tStabilizer{\ofunc}$ and $t\in[0,1]$ we have that
\begin{align*}
  \flow_{t\beta(\odif)} \circ \xi(x) &=
  \flow_{t\beta(\odif)\circ\xi(x)} (\xi(x)) \stackrel{\eqref{equ:skew_symm_flow}}{=\!=\!=}
  \xi\circ \flow_{-t\beta(\odif)\circ\xi(x)}(x) \stackrel{\ref{enum:deform_in_stab:beta_skew_symm}}{=\!=} \\
  &=\xi\circ \flow_{t\beta(\odif)(x)}(x) =
  \xi\circ \flow_{t\beta(\odif)}(x).
\end{align*}
This means that the map $\flow_{t\beta(\odif)}$ belongs to $\tStabilizer{\ofunc}$, whence $H(\odif,t)=(\flow_{t\beta(\odif)})^{-1}\circ \odif \in \tStabilizer{\ofunc}$ as well.

\medskip

Thus it remains to prove~\eqref{equ:-g_xi__g}.
Let $\odif\in\mathcal{A} \cap \tStabilizer{\ofunc}$, so
\begin{align*}
  \odif\circ\xi&=\xi\circ\odif, &
  \odif(x) &= \flow(x,\cov{\gamma}(\odif)(x))
\end{align*}
for all $x\in\OSubSurf$.
Then
\begin{align*}
\odif\circ \xi(x) \ &=  \
\flow\bigr(\xi(x), \cov{\gamma}(\odif)\circ\xi(x) \bigr) \ = \
\flow_{\cov{\gamma}(\odif)\circ \xi(x)} \circ \xi(x) \
\stackrel{\eqref{equ:skew_symm_flow}}{=\!=\!=} \
\xi\circ \flow_{ -\cov{\gamma}(\odif)\circ\xi(x) }(x).
  \\
\xi\circ\odif(x) \ &= \
\xi \circ \flow\bigr(x, \cov{\gamma}(\odif)(x) \bigr) \ = \ \xi\circ \flow_{\cov{\gamma}(\odif)(x)} (x).
\end{align*}
Hence
\[
  \flow_{ -\cov{\gamma}(\odif)\circ\xi(x) }(x)
  = \flow_{ \cov{\gamma}(\odif)(x)} (x) = \odif(x)
\]
for all $x\in\OSubSurf$.
In other words, $-\cov{\gamma}(\odif)\circ\xi$ and $\cov{\gamma}(\odif)$ are shift functions for $\odif$ on $\OSubSurf$.

\ref{enum:skewsym:saddles}
Suppose that each connected component $\Yman$ of $\OSubMan$ contains either a degenerate local extreme or a saddle critical point of $\ofunc$.
Then the shift map on $\OSubSurf_{\Yman}$ is injective, that is any two shift functions for $\odif$ on $\OSubSurf_{\Yman}$ must coincide.
Hence $-\cov{\gamma}(\odif)\circ\xi$ and $\cov{\gamma}(\odif)$ coincide on all of $\OSubSurf$.

\ref{enum:skewsym:gamma_is_zero}
If $\gamma(\odif)\equiv0$ on all of $\OSubMan$, then $\gamma(\odif) = -\gamma(\odif)\circ\xi = 0$ on $\OSubMan$ as well since $\xi(\OSubMan)=\OSubMan$.
Moreover, as $\OSubSurf$ is a $\ofunc$-regular neighborhood of $\OSubMan$, $\OSubMan$ intersects interiors of all connected components of $\OSubSurf$.
It then follows from~\cite[Lemma~6.1(ii)]{Maksymenko:DefFuncI:2014} that $-\cov{\gamma}(\odif)\circ\xi$ and $\cov{\gamma}(\odif)$ coincide on all of $\OSubSurf$.
Lemma~\ref{lm:deform_in_stab} is completed.
\end{proof}

\medskip

Now we can  prove that the inclusions~\eqref{equ:SfNX_SfnbX_SfX} are homotopy equivalences.
Due to the statement~\ref{enum:SfX_tStftX} of Lemma~\ref{lm:SttSt} we can identify groups in~\eqref{equ:SfNX_SfnbX_SfX} with their ``symmetric'' variants, and so it suffices to show that the following inclusions are homotopy equivalences:
\begin{equation}\label{equ:t_SfNX_SfnbX_SfX}
    \tStabilizer{\ofunc,\NtX} \ \subset \
    \tStabilizerNbh{\ofunc,\OSubMan} \ \subset \
    \tStabilizer{\ofunc,\OSubMan}.
\end{equation}

Let $\OSubSurf$ be any $\ofunc$-regular neighborhood of $\NtX$,  $\mathcal{A} = \tStabilizer{\ofunc,\OSubMan}$, and $\gamma:\mathcal{A} \to \Cinfty(\OSubMan,\bR)$ be a constant map into the zero function.
Then for each $x\in\OSubMan$ and $\odif\in\mathcal{A}$ we have that
\[
\flow(x,\gamma(\odif)(x)) = \flow(x,0) = x = \odif(x).
\]
Hence by~\ref{enum:deform_in_stab:skew_symm_Hinv} of Lemma~\ref{lm:deform_in_stab} there exists a homotopy $H:\mathcal{A}\times I \to \Stabilizer{\ofunc}$ such that
\begin{itemize}
\item
$H_0=\id_{\mathcal{A}}$ \  and \ $H_1(\mathcal{A}) \subset \tStabilizer{\ofunc,\OSubSurf}$;
\item
if $\odif\in\mathcal{A}$ is fixed on some $\ofunc$-regular neighborhood of $\OSubMan$ contained in $\NtX$, then so is $H_t(\odif)$ for all $t\in[0,1]$;
\end{itemize}
In other words, $H$ is a deformation of $\mathcal{A} = \tStabilizer{\ofunc,\OSubMan}$ into $\tStabilizer{\ofunc,\OSubSurf}$ which leaves invariant $\tStabilizer{\ofunc,\NtX}$, and $\tStabilizerNbh{\ofunc,\OSubMan}$.
Hence the inclusions~\eqref{equ:t_SfNX_SfnbX_SfX} and therefore~\eqref{equ:SfNX_SfnbX_SfX} are homotopy equivalences.
\end{proof}

\subsection*{Simplification of diffeomorphisms preserving a function by isotopy}
Let $\NSurf$ be a non-orientable compact connected surface, $p:\OSurf\to\NSurf$ be the orientable double covering, and $\xi:\OSurf\to\OSurf$ be an involution without fixed points generating the group $\bZ_2$ of covering transformations.
Let also $\nfunc\in\FSp{\NSurf}$ and $\ofunc = \nfunc\circ p \in \FSp{\OSurf}$.
Since $\OSurf$ is orientable, one can construct a skew-symmetric Hamiltonian like flow $\flow$ on $\OSurf$ for $\ofunc$.

Let $\NSubMan\subset\NSurf$ be a connected $\nfunc$-adopted subsurface and $\NSubSurf\subset\NSurf$ be an $\nfunc$-adopted submanifold.
Denote $\OSubMan = p^{-1}(\NSubMan)$ and $\OSubSurf = p^{-1}(\NSubSurf)$.

Let also $\Stabilizer{\ofunc,\OSubSurf;\OSubMan}$ the subset of $\Stabilizer{\ofunc,\OSubSurf}$ consisting of diffeomorphisms $\odif$ admitting a $\Cinfty$ function $\alpha_{\odif}:\OSubMan\to\bR$ with the following properties:
\begin{enumerate}[itemsep=1ex]
\item
$\odif(x) = \flow(x,\alpha_{\odif}(x))$ for all $x\in\Xman$;
\item
$\alpha_{\odif}=0$ on $\OSubMan\cap\OSubSurf$.
\end{enumerate}

Denote
\[
    \tStabilizer{\ofunc,\OSubSurf;\OSubMan}:= \Stabilizer{\ofunc,\OSubSurf;\OSubMan} \cap \tDiff(\OSurf).
\]
Evidently, we have the following inclusion:
\begin{equation}\label{equ:}
\tStabilizer{\ofunc,\OSubSurf\cup\OSubMan} \ \subset \ \tStabilizer{\ofunc,\OSubSurf; \OSubMan},
\end{equation}
since for each $\odif\in\tStabilizer{\ofunc,\OSubSurf\cup\OSubMan}$ one can set $\alpha_{\odif}\equiv 0$ on $\OSubMan$.

Using isomorphism $s$ from~\eqref{equ:s_iso_Stabs} put
\begin{align*}
    \Stabilizer{\nfunc,\NSubSurf;\NSubMan} & :=  s^{-1}\bigl( \tStabilizer{\ofunc,\OSubSurf;\OSubMan} \bigr), &
    \FolStabilizer{\nfunc,\NSubSurf;\NSubMan} & := \FolStabilizer{\nfunc} \cap \Stabilizer{\nfunc,\NSubSurf;\NSubMan}.
\end{align*}

Then we obviously have the following inclusions:
\begin{align}\label{equ:inclusions_DfVX_SfVX}
    \FolStabilizer{\nfunc,\NSubSurf\cup\NSubMan} &\subset \FolStabilizer{\nfunc,\NSubSurf;\NSubMan},
    &
    \Stabilizer{\nfunc,\NSubSurf\cup\NSubMan} &\subset \Stabilizer{\nfunc,\NSubSurf;\NSubMan}
\end{align}

\begin{corollary}\label{cor:SX_Sidf}{\rm cf.~\cite[Corollary~7.3]{Maksymenko:DefFuncI:2014}.}
Let $\NSurf$ be a compact surface (orientable or not), $\nfunc\in\FSp{\NSurf}$, $\NSubSurf$ be an $\nfunc$-adapted submanifold, and $\NSubMan$ be a connected $\nfunc$-adapted subsurface containing at least one saddle critical point of $\nfunc$.
Then the inclusions~\eqref{equ:inclusions_DfVX_SfVX} are homotopy equivalences.
\end{corollary}
\begin{proof}
Since $\FolStabilizer{\nfunc}$ consists of path components of $\Stabilizer{\func}$, it suffices to show only that the second inclusion is a homotopy equivalence.

For orientable $\NSurf$ this statement coincides with~\cite[Corollary~7.3]{Maksymenko:DefFuncI:2014}.
Let us briefly recall the main steps of its proof.
Since $\NSubMan$ is connected and contains saddle critical points of $\nfunc$, one can show that for every $\ndif\in\Stabilizer{\nfunc,\NSubSurf;\NSubMan}$ the function $\alpha_{\ndif}$ is unique and the correspondence $\ndif\mapsto\alpha_{\ndif}$ is a continuous map $\gamma:\Stabilizer{\nfunc,\NSubSurf;\NSubMan} \to \Cinfty(\NSubMan,\bR)$.
Then a deformation of $\Stabilizer{\nfunc,\NSubSurf;\NSubMan}$ into $\Stabilizer{\nfunc,\NSubSurf\cup\NSubMan}$ can be deduced from \cite[Lemma~7.1]{Maksymenko:DefFuncI:2014} which is the same as Lemma~\ref{lm:deform_in_stab}.

Suppose $\NSurf$ is non-orientable.
Let $p:\OSurf\to\NSurf$ be the orientable double covering, and $\xi:\OSurf\to\OSurf$ be an involution without fixed points generating the group $\bZ_2$ of covering transformations, $\ofunc = \nfunc\circ p \in \FSp{\OSurf}$, and $\fld$ be a skew-symmetric Hamiltonian like vector field for $\ofunc$ on $\OSurf$.
Denote $\OSubMan = p^{-1}(\NSubMan)$ and $\OSubSurf = p^{-1}(\NSubSurf)$.
Then, by the orientable case, the inclusion $\Stabilizer{\ofunc,\OSubSurf\cup\OSubMan} \subset \Stabilizer{\ofunc,\OSubSurf;\OSubMan}$ is a homotopy equivalence.
Moreover, by~\ref{enum:skewsym:saddles} and~\ref{enum:skewsym:gamma_is_zero} of Lemma~\ref{lm:deform_in_stab} the deformation of $\Stabilizer{\ofunc,\OSubSurf;\OSubMan}$ to $\Stabilizer{\ofunc,\OSubSurf\cup\OSubMan}$ preserves $\xi$-symmetric diffeomorphisms, which implies that the inclusion
$\tStabilizer{\ofunc,\OSubSurf\cup\OSubMan} \subset \tStabilizer{\ofunc,\OSubSurf;\OSubMan}$ is a homotopy equivalence as well.

Finally, we have isomorphisms of topological groups~\eqref{equ:s_iso_Stabs}:
\begin{align*}
    &s:\Stabilizer{\nfunc,\NSubSurf\cup\NSubMan} \to \tStabilizer{\ofunc,\OSubSurf\cup\OSubMan}, &
    &s:\Stabilizer{\nfunc,\NSubSurf;\NSubMan} \to \tStabilizer{\ofunc,\OSubSurf;\OSubMan},
\end{align*}
whence the inclusion $\Stabilizer{\nfunc,\NSubSurf\cup\NSubMan} \subset \Stabilizer{\nfunc,\NSubSurf;\NSubMan}$ is a homotopy equivalence as well.
\end{proof}

\section{Functions on annulus}\label{sect:func_on_annulus}
\begin{lemma}\label{lm:cyl:rel_StStIsotId}
Let $\Cyl = S^1\times[0,1]$ and $\func\in\FSp{\Cyl}$.
Then we have the following commutative diagram
\begin{equation}\label{equ:iso:i0StfdC}
\begin{aligned}
\xymatrix{
    0 \ar[r] & \pi_0\FolStabilizer{\func,\partial\Cyl} \ar[rr]^-{j} \ar[d]_-{\cong} &&
    \pi_0\Stabilizer{\nfunc,\partial\Cyl} \ar[d]^-{\psi}_-{\cong}  \ar[r]  &
    \GrpKR{\nfunc,\partial\Cyl} \ar[d]_-{\cong} \ar[r] & 0 \\
    0 \ar[r] & \bZ \times \pi_0\FolStabilizerIsotId{\func,\partial\Cyl} \ar@{^(->}[rr]^-{\id_{\bZ} \times j} &&
    \bZ \times \pi_0\StabilizerIsotId{\func,\partial\Cyl} \ar[r] &
    \GrpKRIsotId{\nfunc,\partial\Cyl} \ar[r] & 0
}
\end{aligned}
\end{equation}
in which $j$ is induced by the inclusion $\FolStabilizer{\func,\partial\Cyl} \subset\Stabilizer{\func,\partial\Cyl}$, the rows are exact, and vertical arrows are isomorphisms.
\end{lemma}
\begin{proof}
Let $\Uman$ be an $\func$-regular neighborhood of $\partial\Cyl$.
Then one can construct a Dehn twist $\tau:\Cyl\to\Cyl$ along $S^1\times0$ supported in $\Uman$ and preserving $\func$, that is $\tau\in\Stabilizer{\func,\partial\Cyl}$, see e.g.~\cite[\S6]{Maksymenko:AGAG:2006}.

It is well known that the mapping class group $\pi_0\Diff(\Cyl,\partial\Cyl)$ is freely generated by the isotopy class of $\tau$, so we have the following sequence of homomorphisms
\[
  \alpha\colon
   \Diff(\Cyl,\partial\Cyl) \xrightarrow{~~~~}
   \frac{\Diff(\Cyl,\partial\Cyl)}{\DiffId(\Cyl,\partial\Cyl)} \equiv \pi_0\Diff(\Cyl,\partial\Cyl)
   \xrightarrow{~~\cong~~} \bZ,
\]
where the first arrow is a natural homomorphism into the mapping class group of $\Cyl$ rel.~$\partial\Cyl$ associating to each $\dif\in\Diff(\Cyl,\partial\Cyl)$ its isotopy class, and the last arrow is an isomorphism.
One can also assume that $q(\tau) = 1$.
Hence the restriction of $\alpha$ to $\Stabilizer{\func,\partial\Cyl}$:
\[\beta = \alpha|_{ \Stabilizer{\func,\partial\Cyl}}: \Stabilizer{\func,\partial\Cyl} \to \bZ\]
is surjective.
Moreover,
\[
\ker(\beta) = \Stabilizer{\func,\partial\Cyl} \cap \ker(\alpha) =
\Stabilizer{\func,\partial\Cyl} \cap \DiffId(\Cyl,\partial\Cyl) =:
\StabilizerIsotId{\func,\partial\Cyl}.
\]
Hence we have the following short exact sequence:
\begin{equation}\label{equ:short_seq_cyl}
1 \to \StabilizerIsotId{\func,\partial\Cyl}
\to  \Stabilizer{\func,\partial\Cyl} \xrightarrow{~~\beta~~} \bZ \to 1,
\end{equation}
and $\beta$ admits a right inverse $\sigma:\bZ\to\Stabilizer{\func,\partial\Cyl}$ defined by $\sigma(k) = \tau^k$, i.e. $\beta\circ\sigma=\id_{\bZ}$.

Due to~\eqref{equ:SfNX_SfnbX_SfX} there is an isomorphism induced by the natural inclusion:
\[
   \pi_0\StabilizerIsotId{\func,\Uman}  \cong \pi_0 \StabilizerIsotId{\func,\partial\Cyl},
\]
whence~\eqref{equ:short_seq_cyl} reduces to the following exact sequence:
\[
1 \to \pi_0\StabilizerIsotId{\func,\Uman}
\to  \pi_0\Stabilizer{\func,\partial\Cyl} \xrightarrow{~~\hat{\beta}~~} \bZ \to 1,
\]
in which $\hat{\beta}$ admits a right inverse $\hat{\sigma}:\bZ\to\pi_0\Stabilizer{\func,\partial\Cyl}$ given by $\hat{\sigma}(k) = [\tau]^k$, $k\in\bZ$.

Since $\tau$ is supported in $\Uman$, it follows that $\tau$ commutes with each $\dif\in\StabilizerIsotId{\func,\Uman}$.
Therefore the subgroups $\pi_0\StabilizerIsotId{\func,\Uman}$ and $\langle [\tau] \rangle \cong\bZ$ mutually commute and generate all the group $\pi_0\Stabilizer{\func,\partial\Cyl}$.
Hence an isomorphism~\eqref{equ:iso:i0StfdC} can be defined by
\[
\psi([\dif]) = \bigl(\beta(\dif),  [\dif \circ \tau^{-\beta(\dif)}] \bigr),
\]
for $\dif\in\Stabilizer{\func,\partial\Cyl}$.

Regarding $\pi_0\FolStabilizer{\func,\partial\Cyl}$ as a subgroup of $\pi_0\Stabilizer{\func,\partial\Cyl}$, one easily checks that $\psi$ maps $\pi_0\FolStabilizer{\func,\partial\Cyl}$ onto $\bZ \times \pi_0\StabilizerIsotId{\func,\partial\Cyl}$, whence left and right vertical arrows in~\eqref{equ:iso:i0StfdC} are isomorphisms.
\end{proof}

\section{Proof of Theorem~\ref{th:pi0S_struct}}\label{sect:th:2}
By Theorem~\ref{th:unique_cr_level} there exists a unique a critical component $\CrComp$ of some level-set of $\nfunc$ such that if $\NSubSurf$ is an $\nfunc$-regular neighborhood of $\CrComp$, and $\YYi{0}, \YYi{1}, \dots, \YYi{n}$ are all the connected components of $\overline{\MBand\setminus\NSubSurf}$ enumerated so that $\partial\MBand\subset\YYi{0}$, then $\YYi{0}$ is an annulus $S^1\times[0,1]$, each $\YYi{k}$, $k=1,\ldots,n$, is a $2$-disk, see Figure~\ref{fig:func_mb_2crpt}.
\begin{figure}[htbp]
\includegraphics[height=3cm]{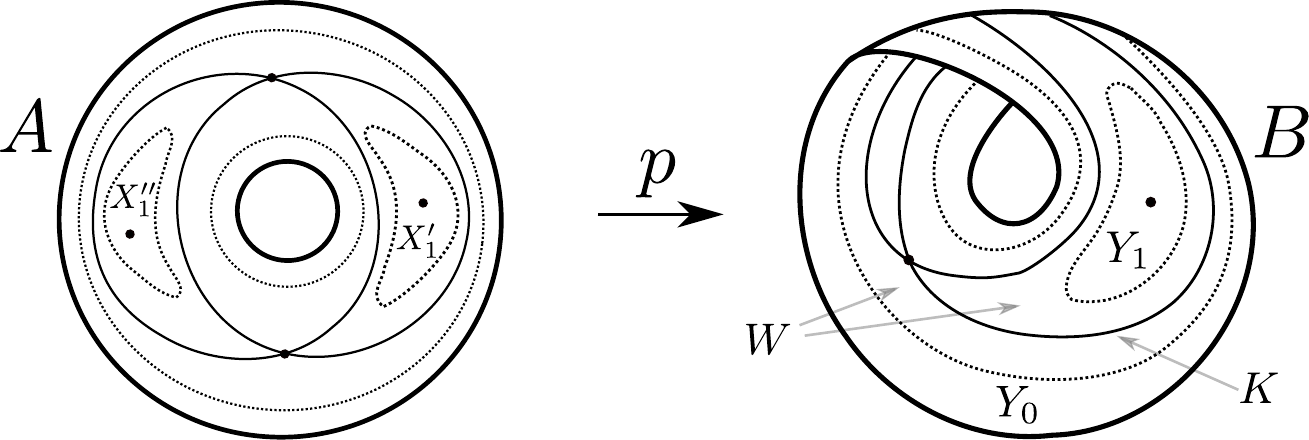}
\caption{}\label{fig:func_mb_2crpt}
\end{figure}
There is also a natural action of $\Stabilizer{\nfunc,\partial\MBand}$ on $\PlCompSet = \{\YYi{1}, \dots, \YYi{n}\} \times\{\pm 1\}$.

\subsection*{Proof that $\Stabilizer{\nfunc,\partial\MBand}/\KerSAct$ freely acts on $\PlCompSet$.}
Let us glue $\partial\MBand$ with a $2$-disk and denote the obtained surface (being therefore a projective plane) with $\hat{B}$.
Let also $\hat{Y}_i$, $i=0,1,\ldots,n$, be the connected component of $\hat{B}\setminus\CrComp$ containing $\YYi{i}$.
Then we have a CW-partition $\Xi$ of $\hat{B}$ whose $0$-cells are critical points of $\nfunc$ belonging to $\CrComp$, $1$-cells are connected components of the complement $\CrComp\setminus\fSing$, and $2$-cells are $\{ \hat{Y}_i \}_{i=0,\ldots,n}$.

Notice that each $\ndif\in\Stabilizer{\nfunc,\partial\MBand}$ extends to a unique \myemph{homeomorphism} $\hat{\ndif}$ of $\hat{B}$ fixed on $\hat{B}\setminus\MBand$.
Moreover, since $\ndif(\CrComp)=\CrComp$, it also follows that $\hat{\ndif}$ is $\Xi$-cellular, i.e. it induces a permutation of cells of $\Xi$.

Suppose $\hat{\ndif}(e)=e$ for some cell $e$ of $\Xi$.
If either
\begin{itemize}
\item $\dim e = 0$ or
\item $\dim e = 1,2$ and $\hat{\ndif}$ preserves its orientation
\end{itemize}
then we will say that $e$ is \myemph{$\hat{\ndif}^{+}$-invariant}.
In particular, $\hat{\ndif}$ has a $\hat{\ndif}^{+}$-invariant cell $\hat{Y}_0$.

Since $\hat{\ndif}$ is also isotopic to $\id_{\hat{B}}$, its Lefschetz number $L(\hat{\ndif}) = \chi(\hat{B}) = 1$.
Then it follows from~\cite[Corollary~5.6]{Maksymenko:MFAT:2010} that
\begin{itemize}
\item either the number of $\hat{\ndif}^{+}$-invariant cells of $\Xi$ is $\chi(\hat{B})=1$, or
\item all cells of $\Xi$ are $\hat{\ndif}^{+}$-invariant.
\end{itemize}

Suppose $\ndif(\YYi{i},+)=(\YYi{i},+)$ for some $i\in\{1,\ldots,n\}$.
Then $\hat{Y}_i$ is $\hat{\ndif}^{+}$-invariant, and so $\hat{\ndif}$ has at least $2 > 1$ $\hat{\ndif}^{+}$-invariant cells.
Hence all cells of $\Xi$ are $\hat{\ndif}^{+}$-invariant, which implies that $\ndif(\YYi{j},+)=(\YYi{j},+)$ for all other $j\in\{1,\ldots,n\}$.
This means that the action of $\Stabilizer{\nfunc,\partial\MBand}/\KerSAct$ on $\PlCompSet$ is free.

\medskip

To construct isomorphism~\eqref{equ:Qf_struct} we need the following two lemmas.

\begin{lemma}\label{lm:Gf_SfdM_W}
$\KerSAct = \Stabilizer{\nfunc,\partial\MBand; \NSubSurf}$.
\end{lemma}
\begin{proof}
Let $\ndif\in\Stabilizer{\nfunc,\partial\MBand;\NSubSurf}$.
Then Corollary~\ref{cor:SX_Sidf} implies that $\ndif$ is isotopic in $\Stabilizer{\nfunc,\partial\MBand;\NSubSurf}$ to a diffeomorphism $\ndif'\in\Stabilizer{\nfunc,\partial\MBand \cup\NSubSurf}$.
Hence $\ndif$ and $\ndif'$ act on $\PlCompSet$ in the same way.
But $\ndif'$ is fixed on $\NSubSurf$, and so on $\partial\YYi{i}$ for all $i=1,\ldots,k$.
Whence $\ndif'$ leaves invariant each $\YYi{i}$ and preserves its orientation, that is $\ndif'\in\KerSAct$.
Therefore so does $\dif$, and thus $\dif\in\KerSAct$ as well, i.e.\! $\Stabilizer{\nfunc,\partial\MBand; \NSubSurf} \subset \KerSAct$.

Conversely, let $\ndif\in\KerSAct$, $p:\Cyl\to\MBand$ be the oriented double covering of $\MBand$, and $\odif = s(\ndif) \in \tStabilizer{\ofunc,\partial\Cyl}$ be a unique lifting of $\ndif$ fixed on $\partial\Cyl$.
Then for each $i=1,\ldots,n$ the preimage $p^{-1}(\YYi{i})$ consists of two connected components $\XXi{i}'$ and $\XXi{i}''$, see Figure~\ref{fig:func_mb_2crpt}.
The assumption that $\ndif(\YYi{i})=\YYi{i}$ and $\ndif$ preserves orientation of $\YYi{i}$ means that $\odif$ leaves invariant both $\XXi{i}'$ and $\XXi{i}''$ and preserves their orientations.
Hence by~\cite[Lemma~7.4]{Maksymenko:DefFuncI:2014}, $\odif$ has a unique shift function $\alpha_{\odif}:p^{-1}(\NSubSurf) \to \bR$.
In other words, $\odif\in\tStabilizer{\ofunc,\partial\Cyl;p^{-1}(\NSubSurf)}$, whence by definition $\ndif=s^{-1}(\odif) = \rho(\odif)\in\Stabilizer{\nfunc,\partial\MBand;\NSubSurf}$.
\end{proof}

\begin{lemma}\label{lm:pi0StfdM_U}
There exists a commutative diagram
\[
\xymatrix{
    \pi_0\FolStabilizer{\nfunc,\partial\MBand; \NSubSurf} \ar[rrr]^-{j} \ar[d]^-{\cong} &&&
    \pi_0\Stabilizer{\nfunc,\partial\MBand; \NSubSurf} \ar[d]^-{\psi}_-{\cong} \\
\bZ \times \prod\limits_{i=0}^{n} \pi_0\FolStabilizerIsotId{\nfunc|_{\YYi{i}},\partial\YYi{i}} \ar[rrr]^-{\id_{\bZ} \times j_0\times \cdots \times j_n} &&&
\bZ \times \prod\limits_{i=0}^{n} \pi_0\StabilizerIsotId{\nfunc|_{\YYi{i}},\partial\YYi{i}}
}
\]
where $j, j_0,\ldots,j_n$ are induced by natural inclusions and the vertical arrows are isomorphisms.
In particular, we get an isomorphism~\eqref{equ:Qf_struct}:
\[
\pi_0\KerSAct \cong \pi_0 \Stabilizer{\nfunc,\partial\MBand; \NSubSurf} \cong
\bZ \times \prod\limits_{i=0}^{n} \pi_0\StabilizerIsotId{\nfunc|_{\YYi{i}},\partial\YYi{i}} =
\bZ \times \prod_{i=0}^{n} \ST{\YYi{i}}.
\]
\end{lemma}
\begin{proof}
As $j$ is a monomorphism, it suffices to construct an isomorphism $\psi$ inducing left arrow.
Due to Theorem~\ref{th:hom_equ} and Corollary~\ref{cor:SX_Sidf} the following inclusions are homotopy equivalences:
\begin{gather*}
    \StabilizerNbh{\nfunc,\partial\MBand\cup\NSubSurf}
    \ \subset \
    \Stabilizer{\nfunc,\partial\MBand\cup\NSubSurf}
    \ \subset \
    \Stabilizer{\nfunc,\partial\MBand;\NSubSurf},
\end{gather*}
whence we need to compute the group $\pi_0 \StabilizerNbh{\nfunc,\partial\MBand\cup\NSubSurf}$ instead.
Further notice that there is a natural isomorphism
\begin{align*}
&\alpha:
\StabilizerNbh{\nfunc,\partial\MBand\cup\NSubSurf} \ \to \
\prod\limits_{i=0}^{n} \StabilizerNbh{\nfunc|_{\YYi{i}},\,\partial\YYi{i}}, &
\alpha(\ndif) &= \bigl( \ndif|_{\YYi{0}},  \ldots,  \ndif|_{\YYi{n}} \bigr),
\end{align*}
for $\ndif\in\StabilizerNbh{\nfunc,\partial\MBand\cup\NSubSurf}$ inducing
isomorphism of the corresponding $\pi_0$-groups:
\begin{align*}
\pi_0\StabilizerNbh{\nfunc,\partial\MBand\cup\NSubSurf} &\cong
 \prod\limits_{i=0}^{n} \pi_0\StabilizerNbh{\nfunc|_{\YYi{i}},\,\partial\YYi{i}}.
\end{align*}
Since $\YYi{0}$ is an annulus, we get from Lemma~\ref{lm:cyl:rel_StStIsotId} that
\[
    \pi_0\StabilizerNbh{\nfunc|_{\YYi{0}},\,\partial\YYi{0}} \ \cong \
    \bZ \times \pi_0\StabilizerNbhIsotId{\nfunc|_{\YYi{0}},\,\partial\YYi{0}}. 
\]
Moreover, $\StabilizerNbh{\nfunc|_{\YYi{i}},\,\partial\YYi{i}} = \StabilizerNbhIsotId{\nfunc|_{\YYi{i}},\,\partial\YYi{i}}$ for all others $2$-disks $\YYi{i}$, $i=1,\ldots,n$, and so
\[
    \pi_0 \StabilizerNbh{\nfunc|_{\YYi{i}},\,\partial\YYi{i}} =
    \pi_0 \StabilizerNbhIsotId{\nfunc|_{\YYi{i}},\,\partial\YYi{i}} \cong
    \pi_0 \StabilizerIsotId{\nfunc|_{\YYi{i}},\,\partial\YYi{i}}.
\]
This gives the required isomorphism $\psi$.

It remains to note that $\psi$ maps $\pi_0\FolStabilizer{\nfunc,\partial\MBand; \NSubSurf}$ (regarded a subgroup of $\pi_0\Stabilizer{\nfunc,\partial\MBand; \NSubSurf}$) onto $\bZ \times \prod\limits_{i=0}^{n} \pi_0\FolStabilizerIsotId{\nfunc|_{\YYi{i}},\partial\YYi{i}}$.
We leave the details to the reader.
\end{proof}

Theorem~\ref{th:pi0S_struct} is completed.


\def\cprime{$'$}

\end {document}